\documentclass{article}

\usepackage{amssymb,amsmath,amsfonts,amsthm}
\usepackage{authblk}

\usepackage{color}

\newtheorem{theorem}{Theorem}[section]
\newtheorem{lemma}[theorem]{Lemma}
\newtheorem{proposition}[theorem]{Proposition}
\newtheorem{corollary}[theorem]{Corollary}

\theoremstyle{definition}
\newtheorem{example}[theorem]{Example}
\newtheorem*{remark}{Remark} 
\newtheorem*{mtheorema}{Main Theorem A}
\newtheorem*{mtheoremb}{Main Theorem B}
\newtheorem{definition}[theorem]{Definition}

\numberwithin{equation}{section}

\newcommand{\Hmm}[1]{\leavevmode{\marginpar{\tiny%
			$\hbox to 0mm{\hspace*{-0.5mm}$\leftarrow$\hss}%
			\vcenter{\vrule depth 0.1mm height 0.1mm width \the\marginparwidth}%
			\hbox to 0mm{\hss$\rightarrow$\hspace*{-0.5mm}}$\\\relax\raggedright #1}}}

\newcommand{\R}{{\mathbb{R}}}
\newcommand{\C}{{\mathbb{C}}}
\renewcommand{\O}{{\mathbb{O}}}
\newcommand{\K}{{\mathbb{K}}}
\renewcommand{\H}{{\mathbb{H}}}

\DeclareMathOperator{\vol}{{\mathsf{vol}}}

\DeclareMathOperator{\Div}{{\mathsf{div}}}

\DeclareMathOperator{\Agmon}{{\mathsf{Agmon}}}

\title{Sharp Hardy-type inequalities for non-compact harmonic manifolds and Damek-Ricci spaces}
\author[1]{Florian Fischer}
\author[2]{Norbert Peyerimhoff}
\affil[1]{Institute of Mathematics, University of Potsdam,  Germany, florifis@uni-potsdam.de}
\affil[2]{Department of Mathematical Sciences, Durham University, Great Britain,
norbert.peyerimhoff@durham.ac.uk}
\date{\today}

\begin{document}

\maketitle

\begin{abstract}
    We show various sharp Hardy-type inequalities for the linear and quasi-linear Laplacian on non-compact harmonic manifolds with a particular focus on the case of Damek-Ricci spaces. Our methods make use of the optimality theory developed by Devyver/Fraas/Pinchover and Devyver/Pinchover and are motivated by corresponding results for hyperbolic spaces by Berchio/Ganguly/Grillo, and Berchio/Ganguly/Grillo/ Pinchover. 


\end{abstract}

\tableofcontents

\section{Introduction}

The aim of this article is to derive Hardy-type inequalities for
the Laplacian on non-compact harmonic manifolds and, in particular, Damek-Ricci spaces, and to discuss applications in the same spirit as  
in Devyver/Fraas/Pinchover~\cite{DFP14},  Devyver/Pinchover~\cite{DP16}, Berchio/Ganguly/Grillo~\cite{BGG17} and Berchio/Ganguly/Grillo/ Pinchover~\cite{BGGP20}. The two latter paper focus mainly on the real hyperbolic space $\H^n = \H^n(\R)$ with $n \ge 3$.
The starting point is the classical Euclidean Hardy inequality
$$ \int_{\R^n} \vert \nabla \phi(x) \vert^2 \, dx \ge \frac{(n-2)^2}{4} \int_{\R^n} \frac{\phi^2(x)}{\vert x \vert^2}\, dx $$
for all smooth and compactly supported functions $\phi \in C_c^\infty(\R^n)$ and $n \ge 3$ (see, e.g., \cite[Cor. 1.2.6]{BEL15} for $p=2$). This inequality was generalised by G. Carron \cite[Prop. 2.1]{Car97} to arbitrary $n$-dimensional Cartan-Hadamard manifolds $X$ (that is, complete and simply connected Riemannian manifolds $(X,g)$ of non-positive curvature) with $|x|^2$ replaced by $d(o,x)^2$, where $o \in X$ is an arbitrary point (pole) and $d$ is the Riemannian distance function on $X$. 

One of the main results in \cite{BGG17} is the following Poincar\'e-Hardy inequality for the $n$-dimensional real hyperbolic spaces $\H^n = \H^n(\R)$, which can be seen as a variation of Carron's Hardy inequality for Cartan-Hadamard manifolds with an additional term in this special case (see \cite[Theorem 2.1]{BGG17}): Given $n \ge 3$ and any point $o \in \H^n$, we have
\begin{multline} \label{eq:PHhyper} 
\int_{\H^n} \vert \nabla \phi(x) \vert^2 dx \ge \lambda_0(\H^n) \int_{\H^n} \phi^2(x) dx + \frac{1}{4} \int_{\H^n} \frac{\phi^2(x)}{d(o,x)^2} dx \\ + \frac{(n-1)(n-3)}{4} \int_{\H^n} \frac{\phi^2(x)}{\sinh^2 d(o,x)}dx 
\end{multline}
for all functions $\phi \in C_c^\infty(\H^n)$. Here $\lambda_0(\H^n)$ denotes the bottom of the spectrum of the (positive) Laplacian $- \Delta = -\Div \circ \nabla$ on $\H^n$, which is known to be $\frac{(n-1)^2}{4}$. If we disregard the extra term $\frac{(n-1)(n-3)}{4} \int_{\H^n} \frac{\phi^2(x)}{\sinh^2 d(o,x)}dx$, the variation can be understood to include an additional $L^2$-norm $\Vert \phi \Vert^2_{\H^n}$ with factor $\lambda_0(\H^n)$ at the expense of the constant in front of the ``Hardy'' term $\int_{\H^n} \frac{\phi^2(x)}{d(o,x)^2}dx$.
Moreover, it is shown in \cite{BGG17} that the constants in inequality \eqref{eq:PHhyper} cannot be improved.

Real hyperbolic spaces $\H^n = \H^n(\R)$ have constant sectional curvature $-1$ and are examples of non-compact \emph{rank-one symmetric spaces}. Besides them, there also exist the (real $2n$-dimensional) complex hyperbolic spaces $\H^n(\C)$,
the (real $4n$-dimensional) quaternionic hyperbolic spaces $\H^n(\H)$ and the (real $16$-dimensional)
Cayley plane $\H^2(\O)$ based on the octonians
$\O$. These spaces comprise (besides the ``simple'' flat space $\R$), the class of all non-compact rank-one symmetric spaces, and they are all Cartan-Hadamard manifolds with sectional curvatures within the interval $[-4,-1]$. 

All these spaces are also non-compact \emph{harmonic manifolds}. A spectral geometric characterization of harmonic manifolds is that they are complete Riemannian manifolds on which all harmonic functions $\phi$ (that is $\Delta \phi = 0$) satisy the \emph{Mean Value Property} (that is $\phi(x_0) = \frac{1}{\vol(S_r(x_0))} \int_{S_r(x_0)} \phi(x)dx$ for all $x_0 \in X$ and spheres $S_r(x_0)$ of radius $r > 0$ around $x_0$).
It was generally assumed (and referred to as the so-called \emph{''Lichnerowicz Conjecture''}) that all simply connected harmonic manifolds should be either Euclidean spaces or rank-one symmetric spaces. This was proved by Z. I. Szab\'{o} \cite{Sz90} in the compact case, and it came as a surprise when E. Damek and F. Ricci discovered in 1992 (see \cite{DR92a,DR92b}) that a whole family of non-compact, non-Euclidean and generally \emph{non-symmetric} homogeneous Riemannian manifolds were indeed also harmonic manifolds (thus disproving this conjecture in the non-compact case). There spaces where studied before (see \cite{ADY,AMPS,ADB,Camp,DR92b,IS10,MV,PS10,Rou03,Rou10} for a selection of papers investigating various of their harmonic analytic properties) and they are nowadays called \emph{Damek-Ricci spaces}. They are solvable extensions $NA$ of $2$-step nilpotent groups $N$ of Heisenberg-type (by a one-dimensional abelian group $A$) with left-invariant metrics. They are associated with a pair of parameters $(p,q)$ which are the dimensions of particular subspaces of the underlying nilpotent Lie algebra of $N$. While the parameters $(p,q)$ do not always uniquely determine the Damek-Ricci space we will use for any Damek-Ricci space with these parameters the notation $X^{p,q}$ (by a slight abuse of notation due to this non-uniqueness).

Damek-Ricci spaces $X^{p,q}$ are homogeneous Cartan-Hadamard manifolds of dimension $n = p+q+1$, with sectional curvatures in the interval $[-1,0]$. The rank-one symmetric spaces $\H^n(\K)$ for the division algebras $\K = \C, \H, \O$ are -- up to scaling of the metric by the constant factor $4$ -- Damek-Ricci spaces with special choices of the parameters $(p,q)$. In fact, we can write (up to the metric scaling factor $4$), $\H^n(\C) = X^{2(n-1),1}$, $\H^n(\H) = X^{4(n-1),3}$, and
$\H^2(\O) = X^{8,7}$. While the real hyperbolic spaces $\H^n$ 
can also be viewed as solvable Lie groups $NA$ with left-invariant metrics, they are not Damek-Ricci spaces since the group $N$ of $\H^n$ is abelian and not $2$-step nilpotent as required for Damek-Ricci spaces. The smallest dimension of a Damek-Ricci space is $4$, and the only possible values $(p,q)$ for Damek-Ricci spaces $X^{p,q}$ are given in the following table with $a \ge 0$ and $b \ge 1$ (see, e.g., \cite[p. 64]{Rou03}). The derivation of these values goes back to \cite[p. 150]{Kap} and is based on the representation theory of Clifford algebras.

\medskip

\noindent
\begin{tabular}{c|c|c|c|c|c|c|c|c}
    $q$ & $8a+1$ & $8a+2$ & $8a+3$ & $8a+4$ & $8a+5$ & $8a+6$ & $8a+7$ & $8a+8$ \\ \hline
    $p$ & $2^{4a+1}b$ & $2^{4a+2}b$ & $2^{4a+2}b$ & $2^{4a+3}b$ & $2^{4a+3}b$ & $2^{4a+3}b$ & $2^{4a+3}b$ & $2^{4a+4}b$
\end{tabular}

\medskip

The Damek-Ricci space of smallest dimension which is not-symmetric is the $7$-dimensional space $X^{4,2}$.
It was shown by J. Heber \cite{Heb06} that any non-compact \emph{homogeneous} harmonic manifold must either be Euclidean, a real hyperbolic space or a (symmetric or non-symmetric) Damek-Ricci space, and it is not known and a challenging open problem whether there exist any further non-compact harmonic manifolds (which must then be necessarily \emph{non-homogeneous}).

Our main result is that the Poincar{\'e}-Hardy inequality \eqref{eq:PHhyper} for $\H^n$ has the following generalization to Damek-Ricci spaces with explicitely given constants:
\begin{mtheorema}
  Let $X^{p,q}$ be a Damek-Ricci space with a pole $o \in X^{p,q}$ and $r = d(o,\cdot)$. Then we have for all $\phi \in C_c^\infty(X^{p,q})$,
\begin{multline} \label{eq:PHDR}
    \int_{X^{p,q}} | \nabla \phi |^2 dx \ge
\lambda_0(X^{p,q}) \int_{X^{p,q}} \phi^2 dx +
\frac{1}{4} \int_{X^{p,q}} \frac{\phi^2}{r^2} dx \\ + \frac{p(p+2q-2)}{16}\int_{X^{p,q}}\frac{\phi^2}{\sinh^2(r/2)} dx+ \frac{q(q-2)}{4}\int_{X^{p,q}}\frac{\phi^2}{\sinh^2 (r)} dx.
\end{multline}
Moreover, the constants on the right hand side of this inequality are optimal and can only be improved at the expense of the other constants. 
\end{mtheorema}
In this paper, we will also discuss two applications of this inequality as well as various variations concerning the constants appearing in this inequality.


Let us briefly discuss the arguments behind this result. They
are based on the fact that, in the case of a simply connected non-compact harmonic manifold $(X,g)$ with a pole $o \in X$, the Laplacian $\Delta$ of a radial function $h(r) = h(d(o,x))$ is again radial and given by
\begin{equation} \label{eq:Lapharmsp} 
\Delta h(r) = h''(r) + \frac{f'(r)}{f(r)} h'(r),
\end{equation}
where $f(r)$ is the volume density of the harmonic manifold. For a general Riemannian manifold $(X,g)$, the volume density $f_o(x) = \sqrt{\det g_{ij}(p)}$ in normal coordinates centered around $o \in X$ is not radial and not independent of the center $o \in X$. In a harmonic manifold $(X,g)$, the volume density is a radial function $f(r)$ and all spheres of the same radius have the same volume. We have
$$ \vol(S_r(x)) = \omega_n f(r), $$
with 
$\omega_n$ being the volume of the unit sphere in $\R^n$. 
Note however, that this does not mean that a harmonic manifold is a \emph{Riemannian model} as described in \cite[Section 4]{BGG17}. In fact, a non-compact harmonic manifold is such a Riemannian model only if it is the Euclidean space $\R^n$ or the hyperbolic space $\H^n$. 

Inequality \eqref{eq:PHDR} is a consequence of the following general Hardy-type inequality for arbitrary non-compact harmonic manifolds:

\begin{mtheoremb} Let $(X,g)$ be a non-compact harmonic manifold
with volume density $f$. Let $o \in X$ be a pole and $r = d(o,\cdot)$. Then we have, for all $\phi \in C_c^\infty(X)$,
\begin{equation} \label{eq:Hharm} 
\int_X |\nabla \phi|^2 dx \ge \frac{1}{4} \int_X \frac{\phi^2}{r^2} dx + \frac{1}{4} \int_X \frac{2f(r)f''(r)-(f'(r))^2}{f^2(r)} \, \phi^2 dx.
\end{equation}
\end{mtheoremb}

Inequality \eqref{eq:PHDR} is then a consequence of \eqref{eq:Hharm} by the following explicit expression of the volume density for a Damek-Ricci space $X^{p,q}$ (see, e.g., \cite[Th{\'e}or{\`e}me 10(ii)]{Rou03}):
\begin{equation} \label{eq:voldens}
f(r) = 2^{p+q}(\sinh(r/2))^{p+q}(\cosh(r/2))^q = 2^p(\sinh(r/2))^p(\sinh r)^q. 
\end{equation}
This implies that
\begin{equation} \label{eq:fprimef} 
\frac{f'(r)}{f(r)} = \frac{p}{2} \coth(r/2) + q \coth r = \frac{p+2q}{2}\coth(r/2) - \frac{q}{\sinh(r)}, 
\end{equation}
and we have for radial functions $h(r)$ in $X^{p,q}$,
$$ \Delta h(r) = h''(r) + \left( \frac{p}{2} \coth(r/2) + q \coth r \right) h'(r). $$
Since $\coth(x) \ge 1/x$ and $\sinh x \ge x$, it follows from \eqref{eq:fprimef} for all $r > 0$ that
\begin{equation} \label{eq:meancurvDR} 
\frac{f'(r)}{f(r)} \ge \frac{p+q}{r}. 
\end{equation}
Moreover, the $L^2$-spectrum of the operator $-\Delta$ on $X^{p,q}$ is given by $\sigma(-\Delta) = [\rho^2,\infty)$ with $\rho = \frac{p+2q}{4}$ (see, e.g., \cite[Remark 2.2(iii)]{PS10}). Therefore, we have $\lambda_0(X^{p,q}) = \frac{(p+2q)^2}{16}$. Since the Cheeger constant of $X^{p,q}$ is given by
$h(X^{p,q}) = 2 \rho = \frac{p+2q}{2}$ (see \cite[Remark 2.2(i)]{PS10}), this means that Cheeger's Inequality $\lambda_0(X^{p,q}) = \frac{h(X^{p,q})^2}{4}$ holds in this case with equality. In fact, this holds for arbitrary non-compact harmonic manifolds $(X,g)$ (see \cite[Corollary 5.2]{PS15}), and the Cheeger constant agrees with other geometric quantities like the (constant) mean curvature of the horospheres or the exponential volume growth (see \cite[Theorem 5.1]{PS15}).

For readers interested in more details about harmonic manifolds and Damek-Ricci spaces, 
we provide this information and a description of the real hyperbolic space $\H^n$ as a solvable Lie group with left-invariant metric in the Appendix.

Moreover, since the seminal works of Hardy and Landau on Hardy-type inequalities more than hundred years ago, these inequalities are studied for the more general quasi-linear case $P>1$ as well. The inequalities discussed before correspond to the linear case $P=2$. We show the related inequalities for $P \ge 2$ in Subsection~\ref{sec:PHardy}. 

The structure of the paper is as follows: In the next section, we state and prove the above Main Theorems. 
Thereafter, we briefly show two famous applications: a version of Heisenberg-Pauli-Weyl's uncertainty principle and a Rellich-type inequality. In Section~\ref{sec:VariMain}, we vary parts of the proof of the main result and get closely connected families of sharp Poincaré-Hardy-type inequalities. The first variation shows the effect of being slightly away from the bottom of the spectrum at the Poincaré-part of the inequality, the second focuses on the weighted version and the third on the $P$-Laplacian version. In Section~\ref{sec:PGreen} we show another Poincar{\'e}-Hardy-type inequality for the $P$-Laplacian using the $P$-Green function. Here the focus is on the asymptotic behaviour of the corresponding Hardy weight. This closes the main part of this paper. In the Appendix, we briefly introduce further information about harmonic manifolds and Damek-Ricci spaces. 
 
\section{A Poincar{\'e}-Hardy-type inequality and applications}\label{sec:Main}

This section is concerned with an analogue of \cite[Theorem 2.1]{BGG17} for all Damek-Ricci spaces. Before we present this result, we first need to introduce some background. The main reference here is \cite{DFP14}. While the concepts hold in more general Lebesgue and Sobolev spaces, we restrict our considerations to the smooth setting.

\subsection{Some background from Optimality Theory} \label{subsec:opttheory}


Henceforth we always assume that $(X,g)$ is a non-compact Riemannian manifold. We are also concerned with the associated Schr\"odinger operators $-\Delta + V$ with potentials $V \in C^\infty(X)$ and their corresponding \emph{energy functionals} $E_V$ on $C_c^\infty(X)$, which are quadratic forms defined via
$$ E_V(\phi) := \int_X |\nabla \phi|^2 + V |\phi|^2 dx. $$
We say that $E_V$ is \emph{non-negative} in $X$ and write $E_V \ge 0$ in $X$, if $E_V(\phi) \ge 0$ for all $\phi \in C_c^\infty(X)$. 

Let $\Omega \subset X$ be a domain (that is a non-empty connected open subset). A function $u \in C^\infty(\Omega)$ is called \emph{solution} (with respect to the operator $-\Delta +V$) if $(-\Delta+V) u = 0$, \emph{subsolution} if $(-\Delta+V) u \le 0$, and \emph{supersolution} if $(-\Delta+V) u \ge 0$ in $\Omega$. Moreover, a function $u\in C^\infty(\Omega)$ is called \emph{(super-)solution near infinity} if there exists a compact set $K\subset X$ such that $(-\Delta +V)u=0$ in $\Omega\setminus K$, (resp.,  $(-\Delta +V)u\geq 0$ in $X\setminus K$). Furthermore, $u \in \Omega$ is a \emph{solution near} $o\in X$, if there is an open set $O\subseteq X$ containing $o$ such that $ (-\Delta +V)u=0$ in $(\Omega \cap O)\setminus \{o\}$.

Let $W \ge 0$ in $X$. 
A Hardy-type inequality then reads as
$$ E_V(\phi) \ge  \int_X W |\phi|^2 dx \qquad \text{for all $\phi \in C_c^\infty(X)$.} $$
One goal in Optimality Theory is -- roughly speaking -- to make $W$ as large as possible with large support (confer \cite[page 6]{Ag82} where this problem was proposed first). With this idea in mind, Devyver, Fraas and Pinchover came up with a definition of an optimal Hardy weight, see \cite[Definitions~2.1, 4.8 and 4.10]{DFP14}. 

\begin{definition}
  Let $(X,g)$ be a non-compact Riemannian manifold and $o \in X$.
  Let $E_V \ge 0$ in $X$ and $W \ge 0$ be a non-trivial function such that the following Hardy-type inequality holds:
  $$ E_V(\phi) \ge  \int_{X \setminus \{o\}} W |\phi|^2 dx \qquad \text{for all $\phi \in C_c^\infty(X \setminus \{o\})$.} $$
  Then $W$ is called an \emph{optimal Hardy weight} of the Schr\"odinger operator $- \Delta + V$ in $X \setminus \{o\}$ if
  \begin{enumerate}
      \item $- \Delta + (V-W)$ is \emph{critical} in $X \setminus \{o\}$, that is, for any $\widetilde W \ge  W$ with $\widetilde W \neq  W$, the Hardy-type inequality
      $$ E_V(\phi) \ge \int_{X \setminus \{o\}} \widetilde W |\phi|^2 dx \qquad \text{for all $\phi \in C_c^\infty(X \setminus \{o\})$} $$
      does not hold. This is equivalent to the following (see e.g.  \cite[Lemma~2.11]{KP20}): there exists a unique (up to a multiplicative constant) positive supersolution to $- \Delta u + (V - W) u = 0$ on $X \setminus \{o\}$. Such a function is also a solution and is called the \emph{(Agmon) ground state} $u_{\Agmon}$.
      \item $- \Delta + (V-W)$ is \emph{null-critical} with respect to $W$, that is, we have that $ u_{\Agmon} \not\in L^2(X \setminus \{o\}, W dx)$.
  \end{enumerate}
\end{definition}

The original definition of an optimal Hardy weight also requires a condition called \emph{optimality at infinity} (which is part (b) in Definition~2.1 of \cite{DFP14}, or see \cite[Definition~2.14]{KP20}). Recently, Kova\v{r}\'{i}k and Pinchover showed in \cite[Corollary~3.7]{KP20} that in our setting, null-criticality implies optimality at infinity. For that reason, it is not necessary to give the definition of optimality at infinity, since it is covered by Condition 2. 

Now we are in a position to present the main result in this section for Damek-Ricci spaces. It is an analogue of \cite[Theorem 2.1]{BGG17} which covers the case of real hyperbolic spaces. Recall that Damek-Ricci spaces do not include real hyperbolic spaces and that their smallest dimension is $4$. Our proof is inspired by the proof given in Section 4 of \cite{BGG17}. 
\begin{theorem}[Poincar\'{e}-Hardy-type inequality on Damek-Ricci spaces]\label{thm:PH}
Let $X^{p,q}$ be a Damek-Ricci space, $o \in X^{p,q}$ be a pole, and $r = d(o,\cdot)$. We have for all $\phi \in C_c^{\infty}(X^{p,q})$,
\begin{multline*}\int_{X^{p,q}} | \nabla \phi |^2 dx -
\lambda_0(X^{p,q}) \int_{X^{p,q}} |\phi|^2 dx \\
\geq \frac{1}{4} \int_{X^{p,q}} \frac{|\phi|^2}{r^2} dx   + \frac{p(p+2q-2)}{16}\int_{X^{p,q}}\frac{|\phi|^2}{\sinh^2(r/2)} dx+ \frac{q(q-2)}{4}\int_{X^{p,q}}\frac{|\phi|^2}{\sinh^2 (r)} dx.
\end{multline*} 
Moreover, the operator $- \Delta + (V-W)$ on $X^{p,q}$ with
\begin{equation*} 
V(r):=  - \lambda_0(X^{p,q}) \quad \text{and} \quad W(r):= \frac{1}{4r^2} + \frac{p(p+2q-2)}{16\sinh^2(r/2)} + \frac{q(q-2)}{4\sinh^2 (r)}
\end{equation*}
is critical in $X^{p,q}\setminus \{o\}$. In particular, there is no 
$\widetilde W(r) \ge \frac{1}{4r^2} +\frac{p(p+2q-2)}{16\sinh^2(r/2)} + \frac{q(q-2)}{4\sinh^2(r)}$, $\widetilde W \neq W$, such that 
$$\int_{X^{p,q}} | \nabla \phi |^2 dx - 
\lambda_0(X^{p,q}) \int_{X^{p,q}} |\phi|^2 dx \\
\geq \int_{X^{p,q}} \widetilde W |\phi|^2 dx
$$
holds true for all $\phi\in C_c^{\infty}(X^{p,q}\setminus \{o\})$.

Furthermore, if we choose as a weight $W_1(r)= \frac{1}{4r^2}$, then this weight is optimal with respect to the operator 
\begin{equation*} 
 -\Delta - \lambda_0(X^{p,q}) -\frac{p(p+2q-2)}{16\sinh^2(r/2)} - \frac{q(q-2)}{4\sinh^2 (r)}.
\end{equation*}
\end{theorem}

The proof of Theorem~\ref{thm:PH} is given in the next subsection, and is based on a result known as Khas’minski\u{i}-type criterion. During the proof, we formulate a Hardy-type inequality for the more general case of non-compact harmonic manifolds. Subsection~\ref{sec:Applications} presents applications of this result: An uncertainty principle and a Poincar\'e-Rellich-type inequality. 
 
\subsection{Proof of the Poincar{\'e}-Hardy-type inequality}\label{sec:proofPH}

To prove Theorem~\ref{thm:PH}, we will use the following result -- known as Khas’minski\u{i}-type criterion. This is a variation of Proposition~6.1 in \cite{DFP14}, confer also with \cite[Lemma~9.2.6]{KPP20}. 
It is in fact an equivalence (see \cite[Theorem~1]{Ancona}) but we only need one direction here.

We start with a non-compact Riemannian manifold $X$ with a chosen point $o\in X$. We refer to $o$ as a \emph{pole} of $X$.

We need some more definitions, confer with \cite[Definition~4.4]{DFP14} and also with \cite[Definition~11.2]{DFP14}:  
Let $K$ be a compact subset of $X\setminus \{o\}$, and let $u$ be a positive function on $X\setminus (\{o\}\cup K)$ which is solution of $(-\Delta+V)u=0$ in $X\setminus (\{o\}\cup K)$. The function $u$ has \emph{minimal growth at infinity} if for every compact $K'\subseteq X\setminus \{o\}$ with smooth boundary such that $K\subseteq \mathrm{int}(K')$ and for every positive supersolution $v\in C(X\setminus (\{o\}\cup K'\cup \partial K'))$ in $X\setminus (\{o\}\cup K')$ with $u\leq v$ in $\partial K'$, we have $u\leq v$ in $X\setminus (\{o\}\cup K')$.

Let $u$ be a positive function defined in a punctured neighbourhood $\Omega $ of the pole $o\in X$ which is a solution of $(-\Delta+V)u=0$ in $\Omega\setminus \{o\}$.  The function $u$ has \emph{minimal growth at $o$} if for every positive supersolution $v$ in a punctured neighbourhood of $o$, there is a constant  $C>0$ such that $u\leq Cv$ in a punctured neighbourhood $\Omega' \subseteq \Omega$ of $o$.

A \emph{global minimal solution} is a positive function on $X\setminus \{o\}$ which is solution of $(-\Delta+V)u=0$ in $X\setminus \{o\}$, and has both, minimal growth at infinity and $o$.

\begin{proposition}[Khas’minski\u{i}-type criterion]\label{prop:khasminski} Let $K$ be a compact subset of $X$. If $u$ and $v_\infty$ are two positive functions on $X\setminus K$ which are solutions of $(-\Delta+V)u=0$ in $X\setminus K$, and if
$$ \lim_{x\to \infty}\frac{u(x )}{v_\infty(x)}=0,$$
then $u$ has minimal growth at infinity.

If $u$ and also $v_0$ are two positive functions in a punctured neighbourhood $\Omega$ of the pole $o$ which are solutions of $(-\Delta+V)u=0$ in $\Omega\setminus \{o\}$, and 
$$ \lim_{x\to o}\frac{u(x )}{v_0(x)}=0,$$
then $u$ has minimal growth at $o$.
\end{proposition}
\begin{proof}
    The first part is Proposition~6.1 in \cite{DFP14}, which is essentially an application of the maximum principle to an exhaustion of the space. The second part follows by a mild adaption of the same argument as in the proof of the cited proposition.
\end{proof}

The proof of Theorem~\ref{thm:PH} uses some general explicit calculations for radial functions. These are extracted in the next lemmata. We mainly work in harmonic manifolds and use the explicit form of the volume density on Damek-Ricci spaces as late as possible.

\begin{lemma} \label{lem:alphabeta}
  Let $(X,g)$ be a non-compact harmonic manifold
  with volume density $f$, $o \in X^{p,q}$ a pole, $r = d(o,\cdot)$, and $\alpha, \beta\in \mathbb{R}$. Then 
  $$ \Phi(r) =  r^{\alpha} f^{\beta}(r)$$
  satisfies on $X\setminus \{o\}$
  $$ \Delta \Phi(r) = \left(\frac{\alpha(\alpha-1)}{r^2}+ \alpha(2\beta +1)\frac{f'(r)}{r f(r)} +\beta \frac{f''(r)}{f(r)}+ \beta^2\frac{(f'(r))^2}{f(r)^2} \right)\Phi(r). $$ 
  
  In particular, if $(\alpha,\beta)= (1/2,-1/2)$, we have on  $X\setminus \{o\}$,
   $$ \Delta \Phi(r) = \left(\frac{(f'(r))^2- 2 f(r)f''(r)}{4 f(r)^2} -\frac{1}{4r^2} \right)\Phi(r). $$ 
\end{lemma}
\begin{proof}
 For simplicity, we will drop the argument $r$ in our computations. Differentiation yields
  $$ \Phi' = \alpha r^{\alpha -1}f^{\beta}+ \beta r^{\alpha} f^{\beta-1}f' $$
  and
  $$ \Phi'' = \alpha (\alpha -1)r^{\alpha-2} f^{\beta}+ 2\alpha\beta r^{\alpha-1} f^{\beta-1}f'+ \beta (\beta -1)r^{\alpha} f^{\beta -2}(f')^2+ \beta r^{\alpha} f^{\beta-1}f''.$$
  Using \eqref{eq:Lapharmsp} and substituting $\Phi$ back into the Laplace equation leads to
  \begin{multline*} 
  \Delta \Phi = \Phi'' + \frac{f'}{f} \Phi' 
     = \left(\frac{\alpha(\alpha-1)}{r^2}+ \frac{\alpha(2\beta +1)f'}{r f}+ \frac{\beta f''}{f} + \frac{\beta^2 (f')^2}{f^2} \right)\Phi. \qedhere
  \end{multline*}
\end{proof}



In the proof of the main results, we want to apply  the Khas’minski\u{i}-type criterion, Proposition~\ref{prop:khasminski}. Therefore, we need to find a second positive solution. This is done next.

\begin{lemma}\label{lem:Phig}
  Let $(X,g)$ be a non-compact harmonic manifold with volume density $f$, $o \in X$ a pole, $r = d(o,\cdot)$, and $\alpha, \beta\in \mathbb{R}$. Furthermore, set 
  $$ \Phi(r) = r^{\alpha} f^{\beta}(r).$$
  Let $h\colon X\setminus \{o\}\to \mathbb{R}$ be a smooth radial function, then on $X\setminus \{o\}$ we have
  \begin{multline*}
      \Delta (\Phi h)(r) \\
      -  \left(\frac{\alpha(\alpha-1)}{r^2}+ \frac{\alpha(2\beta +1)f'(r)}{r f(r)} + \frac{\beta f(r) f''(r)+\beta^2 (f'(r))^2}{f^2(r)} \right) (\Phi(r) h(r)) \\
      = \Phi (r) \left(h''(r) +\left( \frac{2\alpha}{r}+\frac{(2\beta +1)f'(r)}{f(r)} \right) h'(r)\right).
  \end{multline*}
  In particular, the right-hand side vanishes for $(\alpha, \beta, h(r))=(1/2, -1/2, \ln(r))$.
\end{lemma}

\begin{proof}
    In the following, we omit again the argument $r$.  By the product rule, we have 
    $$\Delta (\Phi h)= \Phi \Delta h + h \Delta \Phi + 2 \Phi' h'. $$
    Using this rule and Lemma~\ref{lem:alphabeta}, we obtain
    $$ \Delta (\Phi h) -  \left(\frac{\alpha(\alpha-1)}{r^2}+ \frac{\alpha(2\beta +1)f'}{r f}+ \frac{\beta f''}{f} + \frac{\beta^2 (f')^2}{f^2} \right) (\Phi h)= \Phi \Delta h + 2 \Phi' h'. $$
    Using once more that $h$ is radial, i.e., $\Delta h= h'' + \frac{f'}{f}h'$, and that  $$ \Phi' = \alpha r^{\alpha -1}f^{\beta}+ \beta r^{\alpha} f^{\beta-1}f'= \left(\frac{\alpha}{r}+\frac{\beta f'}{f} \right)\Phi, $$ 
    we obtain
    $$\Phi \Delta h + 2 \Phi' h'=  \Phi \left(h'' +\left( \frac{2\alpha}{r}+\frac{(2\beta +1)f'}{f} \right) h'\right).$$
    This shows the first part of the statement. The other assertion follows by a simple computation.
\end{proof}

Finally, we will need the following lemma, which is a special case of the Agmon-Allegretto-Piepenbrink theorem, see \cite[Theorem~2.3]{PT07} and see also Lemma~\ref{lem:AAP}. We remark that the implication in the lemma below is actually an equivalence, i.e., the non-negativity of the energy functional implies also the existence of a positive (super)solution.

\begin{lemma}[see {\cite[Theorem~1.5.12]{DavHeat}}] \label{lem:davies}
    Let $\Omega \subset X$ be a domain in a Riemannian manifold $(X,g)$, $H := - \Delta + V - W$, and $\Phi \in C^\infty(\Omega)$ be a positive solution of $H \Phi = 0$ on $\Omega$. Then we have, for all $\phi \in C_c^\infty(\Omega)$,
    $$ \int_{\Omega} | \nabla \phi |^2 dx \ge \int_{\Omega} (W-V) \phi^2 dx. $$
\end{lemma}

\begin{proof}
  Set $\phi = \Phi \psi$ with $\psi \in C_c^\infty(\Omega)$. Then we have, by Green's formula,
  $$
      \int_\Omega | \nabla \phi |^2 dx \ge \int_\Omega \langle \nabla(\psi^2 \Phi), \nabla \Phi \rangle dx = -\int_\Omega \psi^2 \Phi (\Delta \Phi) dx = \int_\Omega (W-V) \phi^2 dx. 
  $$
\end{proof}

We link the previous lemmata and the Khas’minski\u{i}-type criterion to obtain the following Hardy-type inequality on general harmonic manifolds, which is our Main Theorem B in the Introduction. 

\begin{theorem}[Hardy-type inequality on harmonic manifolds]\label{thm:HardyHarmonic}
Let $(X,g)$ be a non-compact harmonic manifold with volume density $f=f(r)$, $o \in X$ a pole, and $r = d(o,\cdot)$. 
Then, the Schr\"{o}dinger operator $-\Delta + (V-W)$ on $X$ with
$$ V(r):= \frac{(f'(r))^2-2f(r)f''(r)}{4f^2(r)} \quad \text{and} \quad W(r):= \frac{1}{4 r^2}$$
is critical in $X\setminus \{o\}$ with ground state $\sqrt{r/ f(r)}$.  Moreover, the weight $W(r)=\frac{1}{4 r^2}$ is optimal in $X\setminus \{o\}$ with respect to the operator $-\Delta +V$.

Furthermore, for all $\phi\in C_c^\infty(X)$, we have
\begin{equation} \label{eq:hardymainBineq}
\int_{X }\vert \nabla \phi \vert^2 dx\geq \frac{1}{4} \int_{X} \frac{\phi^2}{r^2}dx   + \frac{1}{4}\int_{X}\frac{2f(r)f''(r)-(f'(r))^2}{f^2(r)}\cdot \phi^2 dx.
\end{equation}
\end{theorem}

\begin{proof} The first part of the proof uses the same ideas as the proof of \cite[Theorem~6.2]{DFP14} doing the necessary changes. The second part is motivated by \cite{BGG17}. Here are the details:

Let us denote $H:=-\Delta +V-W$. By Lemma~\ref{lem:alphabeta} and Lemma~\ref{lem:Phig} applied with $(\alpha, \beta) =(1/2, -1/2)$, we see that $\Phi (r)= r^{1/2}f^{-1/2}(r)$ and $\Phi (r) \cdot \ln (r)$, are solutions on $X\setminus \{o\}$ with respect to $H$. Moreover, we obviously have
 $$\lim_{r\to \infty}\frac{\Phi (r)}{\Phi (r)\ln(r)}= 0. $$
 By the Khas’minski\u{i}-type criterion, Proposition~\ref{prop:khasminski}, we get that $\Phi$ is a positive solution of minimal growth near infinity in $X$.
 
 By the same argument and using $-\Phi(r)\ln(r)$ instead of $\Phi(r)\ln(r)$, we see that $\Phi$ has minimal growth near $o$ in $X\setminus \{o\}$.

 Since $\Phi$ is a solution in $X\setminus\{o\}$ and of minimal growth near infinity and $o$, it is a global minimal solution in $X\setminus\{o\}$. By \cite[Theorems~4.2 and~5.8]{PT08}, it has to be the (Agmon) ground state in $X\setminus\{o\}$ which is the unique positive supersolution of $H$ in $X\setminus\{o\}$. Thus, $H$ is critical in $X\setminus \{o\}$.

Let now $W(r)=1/(2r)^2$. Since 
$$ \Vert \Phi \Vert^2_{L^2(X\setminus \{o\}, W dx)} = \omega_n \int_0^\infty (r^{1/2}f(r)^{-1/2})^2 f(r) \frac{dr}{(2r)^2} = \infty, $$
we conclude that $-\Delta +V$ is null-critical with respect to $W$, and thus $W$ is an optimal Hardy weight of $-\Delta +V$ in $X\setminus \{o\}$.

Inequality \eqref{eq:hardymainBineq}
can be seen as follows: 
Since $\Phi$ is a strictly 
positive solution of $H \Phi = 0$ on $X \setminus \{ o \}$, we obtain the desired inequality on $C_c^\infty(X\setminus \{ o\})$ by Lemma \ref{lem:davies}.
Next we briefly explain how to extend the inequality from $C_c^\infty(X\setminus \{ o\})$ to $C_c^\infty(X)$  (this is a standard argument but we could not find a good reference in the Schr\"{o}dinger operator setting; for the $p$-Laplacian see \cite[Appendix A]{DAD}): By \cite[Theorem~4.5]{PT08}, the capacity of $\{o\}$ vanishes, i.e., there is a sequence $(\phi_j)$ in $C_0^{\infty}(X)$ such that $0\leq \phi_j \leq 1$, $\phi_j=1$ in a neighbourhood of $\{o\}$ and $\phi_j\to 0$ in $D_0(X)=\overline{C_c^{\infty}(X)}^{\vert \cdot \vert_0}$ where $\vert \cdot \vert_0$ denotes the form norm. The latter is defined via $\vert \phi \vert_0^2:= E_V(\phi)+\vert \cdot \vert^2$. In particular, $E_V(\phi_j)\to 0$. We set $\psi_j:=(1-\phi_j)\psi$, where $\psi$ is an arbitrary function in $C_c^{\infty}(X)$. Now it follows $\psi_j\to \psi$ in $D_0(X\setminus \{o\})$ similarly as in \cite[Appendix A]{DAD}): One can use dominated convergence for the potential part and then proceed with mild changes (note that the last step there is just another application of \cite[Theorem~4.5]{PT08}).
\end{proof}

We now turn to Damek-Ricci spaces and explicitly calculate the term $\frac{2ff''-(f')^2}{f^2}$. Later we will need similar calculations, and therefore, we state it here more generally. 

\begin{lemma}\label{lem:Hilfe}
Let $X^{p,q}$ be a Damek-Ricci space, $o \in X^{p,q}$ a pole and $r=d(o,\cdot)$. Let $f(r)$ be the volume density of $X^{p,q}$. Then, for all $a, b\in \R$ and $r>0$, we have
  \begin{multline*}
    a \left(\frac{f'(r)}{f(r)} \right)^2- b \frac{f''(r)}{f(r)}\\
    = 4(a-b)\lambda_0(X^{p,q}) + \frac{q(q(a-b)+b)}{\sinh^2(r)} + \frac{p((a-b)(p+2q)+b)}{4\sinh^2(r/2)}.
\end{multline*}
\end{lemma}
\begin{proof} Recall that $$ f(r) = 2^{p+q}(\sinh(r/2))^{p+q}(\cosh(r/2))^q. $$
    Straightforward computations yield
\begin{align*}
    f'(r)=\frac{f(r)}{2}\bigl( (p+q)\coth (r/2)+ q \tanh (r/2) \bigr).
\end{align*}
and 
\begin{align*}
    f''(r)= \frac{f(r)}{4}\left( \bigl( (p+q)\coth (r/2)+ q \tanh (r/2) \bigr)^2 - \frac{p+q}{\sinh^2(r/2)} + \frac{q}{\cosh^2(r/2)} \right).
\end{align*}
Note, furthermore, that for all $\alpha, \beta\in \mathbb{R}$
\begin{align*}
\frac{\alpha}{\sinh^2(r/2)}- \frac{\beta}{\cosh^2(r/2)}= \frac{\alpha -\beta}{\sinh^2(r/2)}+ \frac{4\beta }{\sinh^2(r)}
\end{align*}
and
\begin{align*}
  \bigl( \alpha\coth (r/2)+ \beta \tanh (r/2) \bigr)^2=(\alpha+\beta)^2+ \frac{4\beta^2}{\sinh^2(r)}+\frac{\alpha^2-\beta^2}{\sinh^2(r/2)}.
\end{align*}
Hence,
\begin{multline*}
    a \left(\frac{f'(r)}{f(r)} \right)^2- b \frac{f''(r)}{f(r)}\\
    = \frac{(a-b)(p+2q)^2}{4} + \frac{q(q(a-b)+b)}{\sinh^2(r)} + \frac{p((a-b)(p+2q)+b)}{4\sinh^2(r/2)}.
\end{multline*}
Using the fact that $\lambda_0(X^{p,q})=(p+2q)^2/16$, the statement follows.
\end{proof}

Note that we obtain from this lemma for the special choice $(a,b) = (1/4,1/2)$:
\begin{equation} \label{eq:ab1412}
    \frac{2f(r)f''(r)-(f'(r))^2}{4(f(r))^2} = \lambda_0(X^{p,q})+ \frac{q(q-2)}{4\sinh^2(r)} + \frac{p(p+2q-2)}{16\sinh^2(r/2)}.
\end{equation}

\begin{lemma}\label{lem:DRf}
Let $X^{p,q}$ be a Damek-Ricci space, $o \in X^{p,q}$ a pole and $r=d(o,\cdot)$. Let $f(r)$ be the volume density of $X^{p,q}$ and
$$ \Phi(r) = r^{1/2}f^{-1/2}(r). $$
Then we have on $X^{p,q}\setminus \{o\}$
$$ - \Delta \Phi(r) - \lambda_0(X^{p,q}) \Phi(r) = \left(\frac{p(p+2q-2)}{16 \sinh^2(r/2)} + \frac{q(q-2)}{4\sinh^2 (r)} + \frac{1}{4r^2} \right)\Phi(r). $$
\end{lemma}

\begin{proof}
The statement of the lemma is a direct application of Lemma \ref{lem:alphabeta} by choosing $(\alpha,\beta) =(1/2, -1/2)$, and of \eqref{eq:ab1412}.
\end{proof}

Finally, we are in a position to prove Theorem~\ref{thm:PH}.

\begin{proof}[Proof of Theorem~\ref{thm:PH}]
    The result follows by applying Lemma~\ref{lem:DRf} and Theorem~\ref{thm:HardyHarmonic}. 
\end{proof} 

\subsection{Applications of the Poincar{\'e}-Hardy-type inequality}\label{sec:Applications}
This subsection is devoted to two well-known applications of Hardy-type inequalities, which can be obtained easily in combination with the Cauchy-Schwarz inequality: an uncertainty-type principle and a Rellich-type inequality. 

\subsubsection{A Heisenberg-Pauli-Weyl's Uncertainty Principle} 

Here, we briefly show a shifted version of the famous Heisenberg-Pauli-Weyl uncertainty principle. It asserts in its classical form, that the position and momentum of a particle can not be determined simultaneously. For further information confer e.g. \cite[Subsection~1.6]{BEL15} for a detailed discussion in the Euclidean space, \cite{KOe09, KOe13, K18Heisenberg} for  Riemannian manifolds, or \cite[Section~3]{BGGP20} for a recent version in the hyperbolic space.

Recall that the dimension of a Damek-Ricci space is at least 4, i.e., $p+q \geq 3$.

\begin{corollary}\label{cor:HPW}
    Let $X^{p,q}$ be a Damek-Ricci space with $q\not\in \{0,2\}$ and $p\neq 0$, $o \in X^{p,q}$ a pole, and $r = d(o,\cdot)$. Then, we have for all $\phi \in C_c^{\infty}(X^{p,q})$,
\begin{multline*}\left(\int_{X^{p,q}} \vert \nabla \phi \vert^2 dx - 
\lambda_0(X^{p,q}) \int_{X^{p,q}} \phi^2\right)\left( \int_{X^{p,q}} g(r)r^2\phi^2 dx\right)\\
\geq  \frac{1}{4} \left(\int_{X^{p,q}} \phi^2 dx\right)^2,
\end{multline*}
with a function $0 < g < 1$. 
\end{corollary}

In fact, the function $g$ in the corollary is explicitely given by
\begin{equation} \label{eqref:g} 
g(r) = \frac{1}{4r^2W(r)}
\end{equation}
with
$$ W(r) := \frac{1}{4r^2}   + \frac{p(p+2q-2)}{16}\frac{1}{\sinh^2(r/2)} + \frac{q(q-2)}{4}\frac{1}{\sinh^2 (r)}. $$
Note that we have $W(r) > 1/(2r)^2$, even for the smallest choice $p \ge 2$ and $q=1$. 

\begin{proof}
    By the Cauchy-Schwarz inequality and the Poincar{\'e}-Hardy-type inequality (Theorem~\ref{thm:PH}), we obtain
\begin{multline*}
   \left(\int_{X^{p,q}} \phi^2 dx\right)^2 \leq \left(\int_{X^{p,q}} W\phi^2 dx\right)\left( \int_{X^{p,q}} \frac{\phi^2}{W} dx\right) \\
   \leq \left(\int_{X^{p,q}} \vert \nabla \phi \vert^2 dx - \lambda_0(X^{p,q}) \int_{X^{p,q}} \phi^2 dx\right)\left( \int_{X^{p,q}} \frac{\phi^2}{W} dx\right).
\end{multline*}
Since $1/W(r) = 4 g(r)r^2 $ we obtain the desired inequality.
\end{proof}

\subsubsection{A Poincar{\'e}-Rellich-type inequality}
Next, we show a shifted Rellich-type inequality. 
For more details on the history and generalisations  of this inequality in different settings, we suggest the papers \cite{BGGP20, KPP21Rellich, KOe09, KOe13} and the monograph \cite{BEL15}, and references therein. 

\begin{corollary}
    Let $X^{p,q}$ be a Damek-Ricci space with $q\not\in \{0,2\}$ and $p\neq 0$, $o \in X^{p,q}$ a pole, and $r = d(o,\cdot)$. Then, we have for all $\phi \in C_c^{\infty}(X^{p,q})$,
    $$ \int_{X^{p,q}} \frac{\phi^2}{16 g(r)r^2} dx
 \leq   \int_{X^{p,q}} g(r)r^2 (-\Delta \phi - \lambda_0(X^{p,q})\phi)^2 dx, $$
with the function $0 < g < 1$ from \eqref{eqref:g}.
\end{corollary}
\begin{proof}
  Let us denote as before,
    $$W(r):=\frac{1}{4r^2}   + \frac{p(p+2q-2)}{16}\frac{1}{\sinh^2(r/2)} + \frac{q(q-2)}{4}\frac{1}{\sinh^2 (r)}. $$
    Moreover, we set $H:= -\Delta - \lambda_0(X^{p,q})$. Then, again by the Poincar{\'e}-Hardy-type inequality (Theorem~\ref{thm:PH}), the Green's formula and the Cauchy-Schwarz inequality, we obtain
\begin{multline*}
 \int_{X^{p,q}}W(r) \phi^2 dx
 \leq \int_{X^{p,q}} \phi H \phi dx \\
 \leq \left( \int_{X^{p,q}} W(r) \phi^2 dx\right)^{1/2}\left(  \int_{X^{p,q}} W^{-1}(r) (H\phi)^2 dx\right)^{1/2}.
\end{multline*}
Since $1/W(r) = 4 g(r)r^2 $ we obtain the desired inequality.
\end{proof}

\section{Variations of the main result}\label{sec:VariMain}
The aim of this section is to show the effect of  perturbations of some parameters. In the first subsection, we will investigate the consequence of changing the function $\Phi$ from $(r/f(r))^{1/2}$ to $(r/f(r))^{1/2}\cdot (r/f^{1/(n-1)}(r))^\gamma$ for some number $\gamma >0$. This will then result in a new family of critical Schrödinger operators, where the effect of $\gamma$ is only visible at the corresponding constants.

In the second subsection, we have a closer look at the weighted version of the Hardy inequality, i.e., we consider the following family of quadratic forms
\[ \int_{X^{p,q}} \frac{\vert \nabla \phi\vert^2}{r^{2\alpha}}dx, \qquad \alpha \geq 0, \phi\in C_c^\infty(X^{p,q}),\]
and show the effect of $\alpha > 0$. In a different perspective, this new weighted energy form is a specific ground state transform of the unweighted energy form.

Finally, in the third subsection, we vary the operator and consider quasi-linear $P$-Laplacians, $P\geq 2$. Here the classical linear case corresponds to $P=2$.

\subsection{A family of Hardy-type inequalities}

In the previous section, we investigated a Hardy-type theorem for the operator $-\Delta - \frac{1}{4r^2}$ on harmonic manifolds. Now, we will use this result to obtain similar inequalities for the operators $-\Delta -\frac{a}{4r^2}$ for $0 \le a \le 1$. The main result of this section reads as follows and is inspired by the corresponding result on hyperbolic spaces, see \cite[Theorem~2.1]{BGGP20}:

\begin{theorem}[Family of Hardy-type inequalities on harmonic manifolds] \label{thm:PHgamma}
Let $(X,g)$ be a non-compact harmonic manifold with volume density $f$, $o \in X$ a pole, $r = d(o,\cdot)$, and dimension $n\geq 3$. Let $h: X \setminus \{o\} \to (0,\infty)$ be a smooth radial function satisfying $h(r) \ge C_1 r$ for all $r \ge 0$ and $h(r) \sim C_2 r$ as $r \to 0$ for some positive constants $C_1, C_2 > 0$. Then, for all $\phi \in C_c^{\infty}(X)$ and all $\gamma \in [0,1/2]$,
\begin{multline} \label{eq:famhardy}
\int_{X} \vert \nabla \phi \vert^2 dx
\geq \frac{ 1-(2\gamma)^2}{4} \int_{X} \frac{\phi^2}{r^2} dx \\  
+ \int_{X}\left( 
\left(\frac{1}{2} \frac{f''}{f} - \frac{1}{4} \left( \frac{f'}{f} \right)^2 \right) + \gamma \left( \frac{h''}{h} + (1+2\gamma)\frac{h'}{rh}  - (1+\gamma)\left(\frac{h'}{h}\right)^2 \right)
\right) \phi^2 dx.
\end{multline} 
Moreover, the Schr\"{o}dinger operator $-\Delta+V-W$ with
$$
    V(r):= - \left(\frac{1}{2} \frac{f''}{f} - \frac{1}{4} \left( \frac{f'}{f} \right)^2 \right) + \gamma \left( \frac{h''}{h} - (1+2\gamma)\frac{h'}{rh}  - (1+\gamma)\left(\frac{h'}{h}\right)^2 \right)
$$
and
$$ W(r):= \frac{ 1-(2\gamma)^2}{4r^2} $$
is critical in $X\setminus \{o\}$ with ground state $r^{1/2+\gamma}f^{-1/2}h^{-\gamma}$.  Furthermore, if we choose as weight $W(r)=\frac{1-(2\gamma )^2}{4 r^2}$, $0 \le \gamma < 1/2$, then $W$ is optimal in $X\setminus \{o\}$ with respect to $-\Delta+V$. 
\end{theorem}

Note that Theorem~\ref{thm:PH} can be recovered from Theorem~\ref{thm:PHgamma} if we set $\gamma = 0$.

We derive this result with the help of the Liouville comparison theorem. For the readers' convenience, we state it here (see also \cite[Theorem~6.3]{BGGP20}).

\begin{proposition}[Liouville comparison principle, Theorem~1.7 in \cite{P07}]\label{prop:Liouville}
Let $\Omega$ be a domain in a non-compact $n$-dimensional Riemannian manifold. Consider two Schr\"{o}dinger operators on $\Omega$ defined via $$H_i = -\Delta + V_i,\qquad i=0,1, $$
with smooth potentials. Furthermore, assume that
\begin{enumerate}
    \item $H_0$ is critical in $\Omega$ with ground state $\Phi$;
    \item $H_1$ is non-negative in $\Omega$ and 
    there is a function $\Psi \in C^\infty(\Omega)$ with $H_1 \Psi \le 0$ with $\Psi_+ \neq 0$, where $\Psi_+(x) = \max\{ \Psi(x),0 \}$.
    \item There exists a constant $C>0$ such that $(\Psi_+)^2\leq C(\Phi)^2$.
\end{enumerate}
Then the operator $H_1$ is critical with ground state $\Psi$.
\end{proposition}

\begin{proof}[Proof of Theorem~\ref{thm:PHgamma}] Let $h\colon X\setminus \{ o \}\to (0,\infty)$ be a smooth radial function satisfying $h(r) \ge C r$ for some positive constant $C > 0$. Applying Lemma~\ref{lem:Phig} with $\tilde{\Phi}=r^\alpha f^\beta$ and  $(\alpha, \beta)= (1/2+\gamma, -1/2)$ yields
\begin{multline*}
      -\Delta (\tilde{\Phi} h^{-\gamma})(r) 
      +  \left(\frac{\gamma^2-1/4}{r^2} 
      +\frac{1}{4}\left(\frac{f'(r)}{f(r)}\right)^2 -\frac{1}{2}\frac{f''(r)}{f(r)} \right) \tilde{\Phi}(r) h^{-\gamma}(r) \\
      = -  \left(\gamma (\gamma +1)  \left(\frac{h'(r)}{h(r)}\right)^2- \gamma \frac{h''(r)}{h(r)} -\gamma (1+2\gamma)\frac{h'(r)}{r h(r)}\right)\tilde{\Phi} (r)h^{-\gamma} (r).
\end{multline*}
Now we apply the Liouville comparison principle with $\Omega = X \setminus \{o\}$,
    $$V_0= -\frac{1}{4r^2} 
      +\frac{1}{4}\left(\frac{f'(r)}{f(r)}\right)^2 -\frac{1}{2}\frac{f''(r)}{f(r)}$$
and, for $\gamma \in [0,1/2]$,
\begin{multline*}
        V_1= \frac{\gamma^2-1/4}{r^2} 
      +\frac{1}{4}\left(\frac{f'(r)}{f(r)}\right)^2 -\frac{1}{2}\frac{f''(r)}{f(r)} \\
      +\gamma (\gamma +1)  \left(\frac{h'(r)}{h(r)}\right)^2-\gamma \frac{h''(r)}{h(r)} -\gamma (1+2\gamma)\frac{h'(r)}{r h(r)}.
\end{multline*}
Let $H_i=-\Delta +V_i$ for $i=0,1$.

 We have shown in Theorem~\ref{thm:HardyHarmonic} that $H_0$ is critical with ground state $\Phi= (r/f)^{1/2}$. Set $\Psi = (r/f)^{1/2}\cdot (r/h)^\gamma >0$ on $\Omega$, that is $\Psi_+ = \Psi$. By construction, we have $H_1\Phi = 0$, and non-negativity of $H_1$ on $\Omega$ follows from Lemma \ref{lem:davies}.

 Since $\Psi / \Phi = (r/h)^{\gamma} \leq 1/C_1$, the Liouville comparison principle implies that $H_1$ is critical with ground state $\Psi$. We can then use the same approximation argument as in the proof of Theorem~\ref{thm:PH} to obtain inequality \eqref{eq:famhardy} for all $\phi\in C_c^\infty(X^{p,q})$ and not only $\phi\in C_c^\infty(\Omega)$.

 Let now $W(r)=(1-(2\gamma)^2)/(2r)^2$ and $0 < \gamma <1/2$. Since $h(r)\sim C_2 r^{n-1}$ as $r\to 0$, we have $\Psi \not\in L^2(\Omega, W dx)$, and $H_1$ is therefore null-critical with respect to $W$. This implies that $W$ is an optimal Hardy weight of $H_1$ in $X\setminus \{o\}$.
\end{proof}

The special choice $h(r) = f^{1/(n-1)}(r)$ leads to the following result which, in the case of Damek-Ricci spaces $X^{p,q}$, can be viewed as Hardy-type improvements of the operators $-\Delta - \lambda$ for $\lambda \le \lambda_0(X^{p,q})$.

\begin{corollary} \label{cor:PHgamma}
    Let $(X,g)$ be a non-compact harmonic manifold with volume density $f$, $o \in X$ a pole, $r = d(o,\cdot)$, and dimension $n\geq 3$. Then, for all $\phi \in C_c^{\infty}(X)$ and all $\gamma \in [0,1/2]$,
\begin{multline} \label{eq:famhardyspec}
\int_{X} \vert \nabla \phi \vert^2 dx
\geq \frac{ 1-(2\gamma)^2}{4} \int_{X} \frac{\phi^2}{r^2} dx + \frac{\gamma(1+2\gamma)}{n-1} \int_X \frac{f'}{f}\cdot \frac{\phi^2}{r} dx +\\ \left( \frac{1}{2} + \frac{\gamma}{n-1} \right) \int_X \left( \frac{f''}{f} - \left( \frac{1}{2} + \frac{\gamma}{n-1} \right) \left( \frac{f'}{f} \right)^2\right) \phi^2 dx.
\end{multline} 
In the particular case of a Damek-Ricci space $X^{p,q}$, we have 
\begin{multline} \label{eq:DRfamHI}
\int_{X^{p,q}} \vert \nabla \phi \vert^2 dx
- \left( 1-\frac{(2\gamma)^2}{(p+q)^2} \right)\lambda_0(X^{p,q}) \int_{X^{p,q}} \phi^2 dx \\
\geq  \frac{ 1-(2\gamma)^2}{4} \int_{X^{p,q}} \frac{\phi^2}{r^2} dx +\frac{\gamma (1+2\gamma)}{p+q}\int_{X^{p,q}}\left( \coth(r/2)-\frac{1}{\sinh(r)}\right)\frac{\phi^2}{r} dx \\  
+ q\left(\frac{1}{2}+\frac{\gamma}{p+q} + q\left(\frac{\gamma^2}{(p+q)^2}-\frac{1}{4} \right) \right)\int_{X^{p,q}}\frac{ \phi^2}{\sinh^2(r)} dx \\
+p\left(\frac{1}{2}+\frac{\gamma}{p+q} + (p+2q)\left(\frac{\gamma^2}{(p+q)^2}-\frac{1}{4} \right)\right)\int_{X^{p,q}} \frac{ \phi^2}{4\sinh^2(r/2)} dx.
\end{multline} 
\end{corollary}

\begin{proof}[Proof of Corollary~\ref{cor:PHgamma}] Let $h = f^{1/(n-1)}$. Let us first verify that $h$ satisfies the conditions required in Theorem \ref{thm:PHgamma}. 
Let $(X,g)$ be a non-compact harmonic manifold of dimension $n \ge 3$ and $H \ge 0$ the (constant) mean curvature of its horospheres. We know from a combination of \cite{RS02} and \cite{Ni05} that non-compact harmonic manifolds with $H=0$ are flat (see also \cite[Section 6]{KP13}). In this case we have  $f(r) = r^{n-1}$ and, therefore $h(r)= r$ and the conditions of the theorem are satisfied with $C_1=C_2=1$. If $H>0$, we have (see \cite[Cor. 2.6]{Kn12})
$\lim_{r \to \infty} \frac{f(r)}{e^{hr}} \ge \left( \frac{n-1}{2H} \right)^{n-1} >0, $
and therefore 
$$ \lim_{r \to \infty} \frac{h(r)}{r} > 0. $$
On the other hand, we have for any $n$-dimensional Riemannian manifold that $\lim_{r \to 0} \frac{f(r)}{r^{n-1}} = 1$, that is, $h(r) \sim r$ as $r \to 0$. Moreover, this implies that
$$ C_1 = \inf_{r > 0} \frac{h(r)}{r} > 0. $$
It is easy to see that
$$ h'= \frac{h}{n-1}\cdot \frac{f'}{f} \quad \text{and} \quad 
        h''= \frac{h}{n-1} \left( \left(\frac{2-n}{n-1}\right)\cdot \left(\frac{f'}{f}\right)^2 + \frac{f''}{f} \right).
$$
Plugging this into \eqref{eq:famhardy} yields \eqref{eq:famhardyspec}. 

Now we derive the inequality for a Damek-Ricci space with $n-1 = p+q$. Using Lemma~\ref{lem:Hilfe} with $a=\left(\frac{1}{2}+ \frac{\gamma}{(p+q)}\right)^2$ and $b=\frac{1}{2}+\frac{\gamma}{p+q}$, we obtain
    \begin{multline*}
      \left(\frac{1}{2}+ \frac{\gamma}{(p+q)} \right)^2\left(\frac{f'(r)}{f(r)}\right)^2- \left(\frac{1}{2}+\frac{\gamma}{p+q}\right)\frac{f''(r)}{f(r)}\\
      = \left( \frac{(2\gamma)^2}{(p+q)^2}-1 \right)\lambda_0(X^{p,q}) 
      + \left(q\left(\frac{\gamma^2}{(p+q)^2}-\frac{1}{4} \right)+\frac{1}{2}+\frac{\gamma}{p+q} \right) \frac{q}{\sinh^2(r)}\\
      + \left( (p+2q)\left(\frac{\gamma^2}{(p+q)^2}-\frac{1}{4} \right)+\frac{1}{2}+\frac{\gamma}{p+q} \right) \frac{p}{4\sinh^2(r/2)}.
    \end{multline*}
Plugging this into \eqref{eq:famhardyspec} yields \eqref{eq:DRfamHI}, finishing the proof of the corollary.
\end{proof}

We want to mention that also other choices of $h$ lead to closely related inequalities. 
In \cite{BGGP20} the choice of $h$ is $h(r) =  \sinh(r)$ which is a natural choice for hyperbolic spaces, and our aim was to generalize that result. 

\subsection{Weighted Hardy-type inequalities}\label{sec:WHR}
Here, we use the method in \cite[Theorem~5.1]{BGGP20}, which is concerned with the hyperbolic space, to obtain a similar result for arbitrary harmonic manifolds $X$. Note again that we obtain the main results in the Introduction by choosing $\alpha=0$ in the following theorem.

\begin{theorem}[Weighted Hardy-type inequality]\label{thm:weightedHardy}
    Let $(X,g)$ be a non-compact harmonic manifold with volume density $f=f(r)$, $o \in X$ a pole, $r = d(o,\cdot)$, and dimension $n\geq 3$. 
    Assume that $n\geq 2(1+\alpha)$ for some $\alpha \geq 0$, then the following  weighted Hardy-type inequality holds true for all $\phi \in C_c^{\infty}(X)$,
    $$\int_X \frac{\vert \nabla \phi \vert^2}{r^{2\alpha}} dx
     \geq \int_X\left(\frac{ 2 f(r)f''(r)-(f'(r))^2}{4 f(r)^2}- \alpha \frac{f'(r)}{rf(r)}+ \frac{4\alpha +1}{4 r^2}\right) \frac{\phi^2}{r^{2\alpha}} dx.$$
     Especially, for such Damek-Ricci spaces $X^{p,q}$, we have for all $\phi \in C_c^{\infty}(X)$ and $p+q-1 \geq 2\alpha$,
     \begin{multline*}
         \int_X \frac{\vert \nabla \phi \vert^2-\lambda_0(X^{p,q}) \phi^2}{r^{2\alpha}} dx
     \geq \int_X\left(\frac{p(p+2q-2)}{16\sinh^2(r/2)} + \frac{q(q-2)}{4\sinh^2 (r)} \right)\frac{\phi^2}{r^{2\alpha}}dx\\
     +  \int_X\left( \frac{\alpha}{r}\left( \frac{q}{\sinh(r)}-\frac{p+2q}{2}\coth (r/2) \right)+ \frac{4\alpha +1}{4 r^2}\right) \frac{\phi^2}{r^{2\alpha}}dx.
     \end{multline*}
\end{theorem}
\begin{proof}
    Given $\phi \in C_c^{\infty}(X\setminus \{ o\})$ and $\Phi$ (which we will specify later), we choose $\psi\in C_c^{\infty}(X\setminus \{ o\})$ such that $\phi=r^\alpha \Phi \psi$. Then, we calculate
    $$\frac{\vert \nabla \phi \vert^2}{r^{2\alpha}}= \vert \nabla \Phi\vert^2\psi^2 + \vert \nabla \psi \vert^2 \Phi^2 +2 \psi\Phi \langle \nabla \Phi, \nabla \psi \rangle +2\alpha\frac{\Phi \psi}{r}(\psi \Phi_r + \Phi \psi_r).$$
    Note that 
    $$\frac{\Phi \psi}{r}(\psi \Phi_r + \Phi \psi_r)=\frac{\phi_r \phi}{r^{2\alpha +1}} - \alpha \frac{\phi^2}{r^{2\alpha +2}}.$$
    Integrating, using that $\Phi \psi^2\in C_c^\infty(X\setminus \{ o \})$ and Green's formula, we obtain
   \begin{multline*}
       \int_X \frac{\vert \nabla \phi \vert^2}{r^{2\alpha}}dx =  \int_X (-\Delta \Phi)\Phi \psi^2 dx+ \int_X  \vert \nabla \psi \vert^2 \Phi^2dx\\
       - 2\alpha^2\int_X \frac{\phi^2}{r^{2\alpha +2}}dx + 2\alpha \int_X \frac{\phi_r \phi}{r^{2\alpha +1}}dx.
   \end{multline*}
    Changing to polar coordinates in the latter integral, using integration by parts, $f\sim r^{n-1}$ as $r\to 0$ and $n-1\geq 2\alpha +1$, and then resubstituting, we get
    
    \begin{multline*}
        2 \int_X \frac{\phi_r \phi}{r^{2\alpha +1}}dx
        = \int_0^\infty \int_{S^{n-1}}\frac{(\phi^2)_r f(r)}{r^{2\alpha +1}} d\omega dr 
        = -\int_0^\infty \int_{S^{n-1}}\left(\frac{ f(r)}{r^{2\alpha +1}}\right)_r \phi^2 d\omega dr \\
        =  (2\alpha+1) \int_X\frac{\phi^2}{r^{2\alpha+2}}dx-  \int_X \frac{f'(r)}{f(r)r^{2\alpha+1}}\phi^2dx.
    \end{multline*}
     Thus, we obtain altogether
     
     $$\int_X \frac{\vert \nabla \phi \vert^2}{r^{2\alpha}} dx
     \geq \int_X (-\Delta\Phi)\Phi \psi^2dx
     +\alpha  \int_X\frac{\phi^2}{r^{2\alpha+2}}dx
     -\alpha \int_X \frac{f'(r)\phi^2}{f(r)r^{2\alpha+1}}dx.$$
    Now, as usual,  let $\Phi = (r/f(r))^{1/2}$, and recall from Lemma~\ref{lem:alphabeta} that 
     $$ -\Delta \Phi(r) = \left(\frac{ 2 f(r)f''(r)-(f'(r))^2}{4 f(r)^2} +\frac{1}{4r^2} \right)\Phi(r). $$ 
     Thus, we get
    $$\int_X \frac{\vert \nabla \phi \vert^2}{r^{2\alpha}} dx
     \geq \int_X\left(\frac{ 2 f(r)f''(r)-(f'(r))^2}{4 f(r)^2}- \alpha \frac{f'(r)}{rf(r)}+ \frac{4\alpha +1}{4 r^2}\right) \frac{\phi^2}{r^{2\alpha}}dx.$$
     Moreover, we can use the same approximation argument as in the proof of Theorem~\ref{thm:PH} to obtain $\phi\in C_c^\infty(X^{p,q})$ and not only $\phi\in C_c^\infty(X^{p,q}\setminus \{ o\})$. 
     This finishes the proof on harmonic manifolds.
     
     On Damek-Ricci spaces $X^{p,q}$, we know $f$ explicitly, and by using \eqref{eq:fprimef} and \eqref{eq:ab1412}  
     the result follows. 
\end{proof}

\subsection{A Poincar{\'e}-Hardy-type inequality for the $P$-Laplacian}\label{sec:PHardy}

In this subsection we want to generalize our Poincar{\'e}-Hardy-type inequality to the quasi-linear $P$-Laplacian for $P \in (1,\infty)$.

Let us first introduce the $P$-Laplacian on a Riemannian manifold $(X,g)$. It is defined as
$$   \Delta_P \phi := {\Div}(| \nabla\phi|^{P-2} \nabla \phi). $$
This quasi-linear second order operator reduces to the classical linear Laplace operator $\Delta$ in the case $P=2$. In the case of a non-compact harmonic manifold $(X,g)$ and a smooth 
radial function $\phi \in C^\infty(X)$ around a pole $o\in X$, the $P$-Laplacian is again radial and given by 
\begin{equation} \label{eq:DeltaP}
\Delta_P \phi(r)=\vert \phi'(r)\vert^{P-2}L_P \phi(r),  
\end{equation}
where
\begin{equation} \label{eq:LP}
L_P \phi(r)= (P-1)\phi''(r) + \frac{f'(r)}{f(r)}\phi'(r).
\end{equation}
Moreover, we will need in Section \ref{sec:PGreen} below that harmonicity of a radial function $\phi \in C^\infty(X \setminus \{o\})$ with nowhere vanishing derivative is equivalent to the condition that $f |\phi'|^{P-2} \phi'$ is constant. This follows directly from the identity
$$ (f |\phi'|^{P-2} \phi')' = f |\phi'|^{P-2} L_P \phi = f \Delta_P \phi. $$

Our aim in this subsection is to prove a quasi-linear version of the Poincar{\'e}-Hardy inequality on Damek-Ricci spaces in the spirit of Theorem~\ref{thm:PH}. This can be seen as an generalisation of Theorem~2.5 in \cite{BDAGG17} on hyperbolic spaces to Damek-Ricci spaces. The proof uses basically the quasi-linear version of the ideas from the linear case. Clearly, the terminology of the Optimality Theory in Subsection \ref{subsec:opttheory} needs to be generalized to the quasi-linear context. However, most of the generalizations are obvious and are not discussed here. We recommend \cite{DP16} for further details and references. Moreover, we want to mention that the same restrictions on $P$ and the dimension in the theorem below also appear in \cite{BDAGG17}.

\begin{theorem}\label{thm:hardyP2}
Let $P\geq 2$, and let $X^{p,q}$ be a Damek-Ricci space of dimension $n = p+q+1 \geq 1+P(P-1)$. Let $o \in X^{p,q}$ be a pole, $r = d(o,\cdot)$, $h:= h(X^{p,q}) = (p+2q)/2$ be its Cheeger constant, and
$$\Lambda_P(X^{p,q}):= \left(\frac{h}{P}\right)^P $$ 
Furthermore, we set 
$$g(r):= \coth (r/2)- \frac{2}{p+2q}\left( \frac{q}{\sinh(r)} + \frac{1}{r}\right). $$
Then, $g(r)\geq \frac{p+q-1}{h\cdot r}$ and $g(r)\to 1$ as $r\to \infty$, and  we have for all $\phi \in C_c^{\infty}(X^{p,q})$,
\begin{multline} \label{eq:PHPineq}
\int_{X^{p,q}} | \nabla \phi |^P dx -
\Lambda_P(X^{p,q}) \int_{X^{p,q}}g(r)^{P-2} |\phi|^P dx \\
\geq \frac{h^{P-2}(P-1)^2}{P^P}\int_{X^{p,q}}\frac{g(r)^{P-2}}{r^2} |\phi|^P dx \\
+ \frac{h^{P-1}(P-2)}{P^P}\int_{X^{p,q}}\frac{g(r)^{P-2}(g(r)+\frac{1}{hr})}{r} |\phi|^P dx \\
+\frac{h^{P-2}}{P^P}\int_{X^{p,q}}g(r)^{P-2}\left(\frac{q^2-qP(P-1)}{\sinh^2(r)} + \frac{p(2h-P(P-1))}{4\sinh^2(r/2)}\right) |\phi|^P dx.
\end{multline} 
Moreover, all integrands on the right hand side of this inequality are non-negative functions.
\end{theorem}

Note again that the special choice $P=2$ in the above theorem leads to our Main Theorem A in the Introduction.

The theorem will be obtained with the help of the Agmon-Allegretto-Piepen\-brink Theorem. For convenience, we state it here in our smooth setting and for a strong solution. Note that it is a generalization of Lemma~\ref{lem:davies}. Moreover, the statement is actually an equivalence but we only state  the direction we need.

\begin{lemma}[see {\cite[Theorem~2.3]{PT07}}]\label{lem:AAP}
    Let $\Omega \subset X$ be a domain in a  Riemannian manifold $(X,g)$, $P\in (1,\infty)$. Let $\Phi \in C^{\infty}(\Omega)$ be a positive supersolution of $-\Delta_P\Phi + (V-W)|\Phi |^{P-2}\Phi = 0$ on $\Omega$. Then we have for all $\phi\in C_c^{\infty}(\Omega)$,
    $$
    \int_\Omega |\nabla \phi |^P dx\geq \int_\Omega (W-V)|\phi |^P dx.
    $$
\end{lemma}

\begin{proof}[Proof of Theorem~\ref{thm:hardyP2}]
Set $\Phi(r) = (r/f(r))^{1/P}$ and let $L_P$ be the operator introduced in \eqref{eq:LP}. It is not difficult to see (omitting the argument $r$),
\begin{multline*}
    L_P\Phi= \left(- \left(\frac{P-1}{P}\right)^2\frac{1}{r^2} - \frac{P-2}{P^2}\frac{f'}{rf} + \frac{P^2-P-1}{P^2} \frac{(f')^2}{f^2}- \frac{P-1}{P}\frac{f''}{f} \right)\Phi \\
    =-\frac{1}{P^2} \left(\frac{(P-1)^2}{r^2} + \frac{(P-2)f'}{rf} + \frac{(f')^2}{f^2}+ P(P-1)\left(\frac{f''}{f}- \frac{(f')^2}{f^2}\right) \right)\Phi.
\end{multline*}
The basic idea is to use Lemma~\ref{lem:AAP}, i.e., we need to show that $\Phi$ is a supersolution with respect to the corresponding $P$-Schrödinger operator, i.e., the value inside the latter parenthesis should be non-negative. Since $f'/f > 0$ (see \cite[Prop. 2.2]{RS02}), it suffices to prove $(f')^2/f^2 + P(P-1)(f''/f-(f')^2/f^2) \ge 0$. Once this is shown, the lemma provides us with a particular Hardy inequality. However, it seems that we need more knowledge about $f$ to prove non-negativity of this term. Since our Riemannian manifold is the Damek-Ricci space $X^{p,q}$ with the explicit expression \eqref{eq:voldens} for the volume density $f$, we conclude from Lemma~\ref{lem:Hilfe} by setting $a=1-P(P-1)$ and $b=-P(P-1)$ and the relation $4\lambda_0(X^{p,q})=h^2$ that
\begin{multline}\label{eq:PAndF}
    \frac{(f')^2}{f^2}+ P(P-1)\left(\frac{f''}{f}- \frac{(f')^2}{f^2}\right)=\\
    h^2 + \frac{q(q-P(P-1))}{\sinh^2(r)}+ \frac{p(p+2q-P(P-1))}{4\sinh^2(r/2)}.
\end{multline}
Using the assumption $n-1=p+q\geq P(P-1)$, we estimate
\[ \frac{q(q-P(P-1))}{\sinh^2(r)}+ \frac{p(p+2q-P(P-1))}{4\sinh^2(r/2)} \ge pq \left(\frac{1}{4\sinh^2(r/2)} - \frac{1}{\sinh^2(r)}\right) \ge 0, \]
where we used $\sinh(r)=2\sinh(r/2)\cosh(r/2)$ and $\cosh(r/2) \geq 1$ for all $r\in \R$ in the last inequality. This shows on the one hand that the last integrand on the right hand side of \eqref{eq:PHPineq} is non-negative, and also that
\[\frac{(f')^2}{f^2}+ P(P-1)\left(\frac{f''}{f}- \frac{(f')^2}{f^2}\right) \geq h^2 > 0. \]



Hence, for $P(P-1)\leq n-1$ and $P\geq 2$, $\Phi$ is a positive supersolution. By Lemma~\ref{lem:AAP}, we conclude the non-negativity of the corresponding energy functional.

To finish the proof of the theorem, we only need to analyse $\Delta_P \Phi(r) = \vert \Phi'(r)\vert^{P-2}L_P\Phi(r)$. 
Note that $$\vert \Phi'(r)\vert^{P-2}= \left(\frac{\Phi}{P}\right)^{P-2}\left| \frac{1}{r}- \frac{f'}{f}\right|^{P-2}.$$
One could use \eqref{eq:meancurvDR} and estimate $| \frac{1}{r}- \frac{f'}{f}| = \frac{f'}{f}-\frac{1}{r} \geq \frac{p+q-1}{r}$ for a simplification. But we want to take the precise term instead and write with the help of \eqref{eq:fprimef},  
$$ \frac{f'}{f}-\frac{1}{r}= \frac{p+2q}{2}\coth (r/2)- \frac{q}{\sinh (r)}- \frac{1}{r}= \frac{p+2q}{2}g(r)=h\cdot g(r),$$
with $g$ and $h=h(X^{p,q})$ defined as in the theorem.
This yields the estimate $g(r) \ge \frac{p+q-1}{h \cdot r}$ in the theorem as well as the asymptotics $g(r) \to 1$ as $r \to \infty$.

Substituting \eqref{eq:PAndF} into
\begin{multline*}
    (-\Delta_P)\Phi=  \frac{h^{P-2}}{P^P}g^{P-2}\cdot \\
    \left(\frac{(P-1)^2}{r^2} + \frac{(P-2)f'}{rf} + \frac{(f')^2}{f^2}+ P(P-1)\left(\frac{f''}{f}- \frac{(f')^2}{f^2}\right) \right)\Phi^{P-1}
\end{multline*}
and using $\frac{f'}{f} = h(g+\frac{1}{hr})$, the inequality of the theorem follows from Lemma~\ref{lem:AAP} and the usual approximation argument to extend it to all functions $\phi \in C_c^\infty(X^{p,q})$.
\end{proof}

\begin{remark}
 It follows from the proof that inequality \eqref{eq:PHPineq} in the theorem can be simplified by replacing $g(r)$ by the easier lower bound $\frac{p+q-1}{hr} > 0$,
 but we expect that the corresponding $P$-Schrödinger operator in the theorem is critical on $X^{p,q}\setminus \{ o\}$, which would be lost by this simplification.
\end{remark}



\section{Another Poincar{\'e}-Hardy-type inequality for the $P$-Laplacian}\label{sec:PGreen}

In this section we derive the Green function for the $P$-Lapacian and follow the general arguments given in \cite[Section 5.1]{BMR} (see also \cite{BDAGG17} for the case of the real hyperbolic space) to obtain a Hardy weight for the $P$-Laplacian on non-compact harmonic manifolds, which is optimal under a certain condition of the dimension. Similarly as in \cite{BMR}, we also study asymptotics of this Hard weight at the pole.  

\subsection{The Green function of the P-Laplacian on a non-compact harmonic manifold}

The $P$-Green function $G = G_P: X \times X \setminus \{(x,x) \mid x \in X\}\to \mathbb{R}$ is characterised by the following conditions (see, e.g., \cite{Ku99}):
\begin{itemize}
    \item[(a)] $\Delta_{P,x} G(x,y) = 0$ for all $ x \neq y$,
    \item[(b)] $G(x,y) \ge 0$ for all $x \neq y$,
    \item[(c)] For all $y \in X$, we have $\lim_{x \to \infty}G(x,y) = 0$,
    \item[(d)] If $1 < P \le n$: For all $y \in X$, we have $\lim_{x \to y} G(x,y) = \infty$,
    \item[(e)] For all $\phi \in C_c^\infty(X)$
    and all $y \in X$, we have 
    $$ \int_X \left\langle | \nabla_x G(x,y) |^{p-2} \nabla_x G(x,y), \nabla \phi(x)\right\rangle dx = \phi(y), $$
    that is, $- \Delta_{P,y} G(x,\cdot) = \delta_y$ in the sense of distributions.
\end{itemize}
In the case of a harmonic manifold, we can fix a point $o \in X$ and consider the Green function as a function of the radius, that is, we can write $G(d(x,o)) = G(x,o)$. The Green function of the $P$-Laplacian of the Euclidean space $\mathbb{R}^n$ for $1 < P < n$ is given by 
$$ G(r) = \frac{P-1}{(n-P)} \omega_n^{-1/(P-1)} r^{-\frac{n-P}{P-1}}, $$ 
where $\omega_n$ is the volume of the unit sphere $S^{n-1}$ in $\mathbb{R}^n$. Let us now calculate the Green function of the $P$-Laplacian of a non-compact non-Euclidean harmonic manifold $X$ of dimension $n$ with density function $f$. Recall from Subsection \ref{sec:PHardy} that $P$-harmonicity of a radial function $G \in C^\infty(X\setminus\{o\})$ with nowhere vanishing derivative is given when $f|G'|^{P-2}G'$ is constant. Assuming $G' < 0$ and $\lim_{r \to \infty} G(r) = 0$, and setting 
$$ f |G'|^{P-2}G' = - f |G'|^{P-1} \cong -\beta^{P-1}, $$
for some constant $\beta > 0$, we find by integration
\begin{equation} \label{eq:Gharm} 
G(r) = \beta \int_r^\infty \frac{dt}{(f(t))^{1/(P-1)}}. 
\end{equation}
Since $X$ is non-flat, the mean curvature $h$ of its horospheres must be strictly positive (see, e.g., \cite[Corollary 2.8]{KP13}). This result is a consequence of the fact that $X$ cannot have polynomial volume growth by \cite[Theorem 4.2]{RS02} and that the density function $f$ of $X$ is an exponential polynomial by \cite{Ni05}. Therefore, $X$ must have exponential volume growth and the integral on the right hand side of \eqref{eq:Gharm} is finite for all $r > 0$. Let us study
$\lim_{r \to 0} G(r)$ for $G(r)$ defined in \eqref{eq:Gharm}. It follows from \cite[p. 82]{Willm93} for arbitrary Riemannian manifolds that
\begin{equation} \label{eq:frto0} 
f(t) = t^{n-1}\left( 1 - \frac{s(o)}{6n} t^2 + O(t^4) \right) \qquad \text{for $t \to 0$,} 
\end{equation}
where $s(o)$ is the scalar curvature of $X$ at $o \in X$. Consequently, we can find $r' > 0$ and $0 < c_1 < c_2$ such that, for $0 < r < r'$,
\begin{equation} \label{eq:c1c2est} 
c_1 \int_r^{r'} \frac{dt}{t^{(n-1)/(P-1)}} \le  \int_r^{r'} \frac{dt}{(f(t))^{1/(P-1)}} 
\le c_2 \int_r^{r'} \frac{dt}{t^{(n-1)/(P-1)}}, 
\end{equation}
which implies $\lim_{r \to 0} G(r) = \infty$ in the case $1 < P \le n$. The Ansatz \eqref{eq:Gharm} implies
$$ G'(r) = - \frac{\beta}{(f(r))^{1/(P-1)}}, $$
which shows that the assumption $G' < 0$ was justified. The arguments so far and the explicit expression for $G$ imply that $G$ satisfies conditions (a), (b), (c) and (d). 

It remains to verify condition (e). 
It is known that non-compact harmonic manifolds do not have conjugate points, since simply connected harmonic manifolds with conjugate points are Blaschke manifolds by Allamigeon's Theorem (see \cite[Chapter 6F]{Besse78} or \cite[Section 5.1]{Krey10}). We can therefore use global geodesic polar coordinates $(r,\theta) \in (0,\infty) \times S^{n-1} \to X \setminus \{y\}$ around the pole $y \in X$. For simplicity of notation, we switch between points of $X$ and their representations in polar coordinates. We have $\nabla_x G(x,y) = G'(r) \xi(r,\theta)$, where $r = d(x,y)$ and $\xi$ defined on $X \setminus\{y\}$ is the outward unit normal vector field of concentric spheres around $y$. Consequently, we have
\begin{multline*} 
\left\langle | \nabla_x G(x,y) |^{P-2} \nabla_x G(x,y), \nabla \phi(x)\right\rangle = |G'(r)|^{P-2} G'(r) \frac{\partial}{\partial r} \phi(r,\theta) \\ = 
- \frac{\beta^{P-1}}{f(r)} \frac{\partial}{\partial r} \phi(r,\theta).
\end{multline*}
This implies
\begin{multline*}
\int_X \left\langle | \nabla_x G(x,y) |^{P-2} \nabla_x G(x,y), \nabla \phi(x)\right\rangle dx \\ = - \int_{S^{n-1}} \int_0^\infty  \frac{\beta^{P-1}}{f(r)} \frac{\partial}{\partial r} \phi(r,\theta) f(r) dr d\theta = \beta^{P-1} {\rm{vol}}(S^{n-1}) \phi(y).
\end{multline*}
This shows that (e) is satisfied if we choose $\beta = 1/{\omega_n^{1/(P-1)}}$. Therefore the Green function of the $P$-Laplacian of an $n$-dimensional non-compact harmonic manifold is given by
$$ G(d(x,y)) = \int_{d(x,y)}^\infty \frac{dt}{(\omega_n f(t))^{1/(P-1)}}. $$
Moreover, we see from estimate \eqref{eq:c1c2est}
that, in the case $n < P < \infty$, the limit
$\gamma := \lim_{r \to 0} G(r)$ is a finite positive real number. Note that this expression agrees with the boxed formula in \cite[p. 51]{KP13} for the standard Laplace operator, that is, the case $P=2$.

\subsection{A Poincar\'e-Hardy-type inequality for the $P$-Laplacian based on the $P$-Green function}

The following theorem is an extension of \cite[Proposition~1.1]{BDAGG17} from real hyperbolic spaces to arbitrary non-flat harmonic manifolds. The proof follows closely the one given in \cite{BDAGG17} which, in turn, is based on the arguments given in \cite[Section 5.1]{BMR}.


\begin{theorem} \label{thm:hardyPgreen}
    Let $X$ be a non-compact non-Euclidean harmonic manifold of dimension $n$, $o\in X$ a pole, and $r=d(o,\cdot)$. Let $1 < P < \infty$ and $G(r)$ be the corresponding $P$-Green function. Then 
    we have
    $$
    W:=\left(\frac{P-1}{P}\right)^P \left| \frac{\nabla G}{G}\right|^P \ge \Lambda_P := \left( \frac{h}{P} \right)^P.   $$
    In the case $1 < P \le n$, $W$ is an optimal Hardy weight of $-\Delta_P$ in $X\setminus \{o\}$.
    
    For all $1 < P < \infty$ and all
    $\phi \in C_c^\infty(X)$,
    we have
    \begin{equation} \label{eq:greenhardy} 
    \int_X | \nabla \phi |^P dx - \Lambda_P \int_X |\phi|^P dx \ge \int_X \widetilde W |\phi|^P dx \end{equation}
    with $\widetilde W = W - \Lambda_P \ge 0$.
    Moreover,
    we have the following asymptotics, 
    \begin{align} \label{eq:asr0}
        \widetilde W(r) \sim \begin{cases} 
         \left(\frac{n- P}{P} \right)^P r^{-P} &\quad \text{for } 1< P< n \\ 
        \left(\frac{P-1}{P}\right)^P|r\log(r)|^{-P} &\quad \text{for } P= n \\
        C_{P,f}\,\, r^{- \frac{P(n-1)}{P-1}}  &\quad \text{for } n< P< \infty \end{cases}  \qquad \text{as } r\to 0,
    \end{align}
    where $C_{P,f}= \left(\frac{P-1}{P}\right)^P\left(\int_0^\infty \frac{dt}{( f(t))^{1/(P-1)}}\right)^{-P}$. We also have $\lim_{r \to \infty} \widetilde W(r) = 0$. 
\end{theorem}

\begin{proof}
    Let us first prove the inequality $W(r) \ge \Lambda_P$. Let $f$ be the density function of $X$. We know from \cite[Section 2]{RS03} that the quotient $f'(r)/f(r)$ is monotone descreasing with $h = \lim_{r \to \infty} f'(r)/f(r)$. Since $X$ is non-flat, we have $h > 0$. Using $f'/f \ge h$, we obtain
    $$ \int_r^\infty \frac{dt}{(f(t))^{1/(P-1)}} \le \frac{1}{h} \int_r^\infty \frac{f'(t)dt}{(f(t))^{1+1/(P-1)}} = \frac{P-1}{h} \left( f(r) \right)^{-1/(P-1)}. $$
    This implies that
    $$ \left\vert \frac{G'(r)}{G(r)} \right\vert =  \left((f(r))^{1/(P-1)} \int_r^\infty \frac{dt}{(f(t))^{1/(P-1)}}  \right)^{-1} \ge \frac{h}{P-1}, $$
    and therefore
    $$ W(r) = \left(\frac{P-1}{P}\right)^P \left\vert \frac{G'(r)}{G(r)} \right\vert^P \ge \left( \frac{P-1}{P} \right)^P \left( \frac{h}{P-1}\right)^P = \Lambda_P. $$
    Applying \cite[Theorem~1.5(1)]{DP16}, we conclude that $W$ is an optimal Hardy weight of $- \Delta_P$ in $X \backslash \{o\}$ in the case $1 < P \le n$. 
    
    Choosing $\rho = G$ in
    \cite[Theorem 2.1]{DAD},
    we conclude for all $\phi \in C_c^\infty(X_0)$ that
    $$ \int_{X_0} | \nabla \phi |^P dx \ge \int_{X_0} W | \phi|^P dx $$
    for $X_0 = X \setminus \{o\}$ if $1 < P \le n$
    and $X_0 =X$ if $P > n$.
    Since $\{o\} \subset X$ is a compact set of zero $p$-capacity, we can apply \cite[Corollary 2.3]{DAD} to extend the inequality to the whole manifold $X$ in the case $1 < P \le n$. This shows inequality \eqref{eq:greenhardy}. 
    
    It remains to prove the asymptotics. Let us begin with the asymptotics as $r \to 0$. Note that $\widetilde W(r) \sim a(r)$ is equivalent to $W(r) \sim a(r)$ in the case $\lim_{r \to 0} a(r) = \infty$, and it suffices to prove the asymptotics in \eqref{eq:asr0} for $W$ instead of $\widetilde W$.
    
    Let us first consider the case $P > n$. Using the estimate \eqref{eq:c1c2est}, we conclude that
    $$ \int_0^r \frac{dt}{(f(t))^{1/(P-1)}} = O(r^{\frac{P-n}{P-1}}) \quad \text{as $r \to 0$,} $$
    and therefore, employing \eqref{eq:frto0},
    \begin{eqnarray*} 
    \left\vert \frac{G'(r)}{G(r)} \right\vert &=& \left[ \left( \int_0^\infty \frac{dt}{(f(t))^{1/(P-1)}}\right)^{-1} + O(r^{(P-n)/(P-1)}) \right] (f(r))^{-1/(P-1)} \\
    &\sim& \left( \int_0^\infty \frac{dt}{(f(t))^{1/(P-1)}}\right)^{-1} r^{-(n-1)/(P-1)}\quad \text{as $r \to 0$.}
    \end{eqnarray*}
    This shows, in the case $P > n$, that
    $$ W(r) = \left(\frac{P-1}{P}\right)^P \left\vert \frac{G'(r)}{G(r)} \right\vert^P \sim C_{P,f} r^{-P(n-1)/(P-1)} \quad \text{as $r \to 0$.} $$
    Now, we turn to the case $1 < P < n$. We have $\lim_{r \to 0} G(r) = \infty$ by property (d) of the $P$-Green function and,
    by using \eqref{eq:frto0},
    $$ \lim_{r \to 0} rG'(r) = - \frac{1}{\omega^{1/(P-1)}} \lim_{r \to 0} \frac{r}{r^{(n-1)/(P-1)}} = - \infty. $$
    We can apply L'H\^opital and obtain
    $$ \lim_{r \to 0} \frac{rG'(r)}{G(r)} = 1 + \lim_{r \to 0} \frac{rG''(r)}{G'(r)} = 1 - \frac{1}{P-1} \lim_{r \to 0} \frac{rf'(r)}{f(r)}. 
    $$
    Note that $f'(r)/f(r)$ is the mean curvature of metric spheres of radius $r$, and it is well-known for arbitrary Riemannian manifolds that $f'(r)/f(r) \sim (n-1)/r$ as $r \to 0$ (in the case of a harmonic manifold, we have $rf'(r)/f(r) = {\rm{trace}}(C)$ with $C$ in \cite[6.33]{Besse78} and the statement follows from \cite[6.36]{Besse78}). Consequently, we have
    $$ \frac{G'(r)}{G(r)} \sim \frac{P-n}{P-1} \cdot \frac{1}{r} \quad \text{as $r \to 0$.} $$
    From this we conclude that, in the case $1 < P < n$,
    $$ W(r) = \left(\frac{P-1}{P}\right)^P \left\vert \frac{G'(r)}{G(r)} \right\vert^P \sim \left( \frac{n-P}{P} \right)^P r^{-P} \quad \text{as $r \to 0$.} $$
    Let us now consider the remaining case $P=n$. Similarly as before we have $\lim_{r \to 0} r \log r G'(r) = - \infty$ and L'H\^opital yields
    \begin{multline*} 
    \lim_{r \to 0} \frac{r\log r G'(r)}{G(r)} = 1 + \lim_{r \to 0} \log r \frac{G'(r)+rG''(r)}{G'(r)} \\ = 1 + \lim_{r \to 0} \log r \left( 1 - \frac{1}{P-1} \frac{rf'(r)}{f(r)} \right) = 1 + \lim_{r \to 0} \log r \left( \frac{P-n}{P-1} + O(r^2) \right). 
    \end{multline*}
    Here we used the slightly stronger fact that $r f'(r)/f(r) = (n-1)+ O(r^2)$ as $r \to 0$ (this follows, e.g., from \cite[Lemma 12.2]{GV79}, where an explicit Taylor expansion is given up to the sixth term). Since we assume $P=n$, we conclude that
    $$ \frac{G'(r)}{G(r)} \sim \frac{1}{r \log r} \quad \text{as $r \to 0$,} $$
    and therefore, in the case $P=n$,
    $$ W(r) = \left(\frac{P-1}{P}\right)^P \left\vert \frac{G'(r)}{G(r)} \right\vert^P \sim \left( \frac{P-1}{P} \right)^P |r \log r|^{-P} \quad \text{as $r \to 0$.}$$
    
   Finally, the proof of $\lim_{r \to \infty} \widetilde W(r) = 0$ is a straightforward application of L'H\^opital. We have
   $$ \lim_{r \to \infty} \frac{G'(r)}{G(r)} = -  \lim_{r \to \infty} \frac{1/(f(r))^{1/(P-1)}}{\int_r^\infty \frac{dt}{(f(t))^{1/(P-1)}}} = - \frac{1}{P-1} \lim_{r \to \infty} \frac{f'(r)}{f(r)} = - \frac{h}{P-1}, $$
   and therefore
   $$ \lim_{r \to \infty} W(r) = \left( \frac{P-1}{P} \right)^P \lim_{r \to \infty} \left| \frac{G'(r)}{G(r)} \right|^P = \left( \frac{P-1}{P} \right)^P \left( \frac{h}{P-1} \right)^P = \Lambda_P. $$
   \end{proof}
   
\begin{remark}
   It follows from \cite[Theorem 1.5(2)]{DP16} that, in the case $n < P < \infty$, an optimal Hardy weight of $-\Delta_P$ in $X \setminus \{o\}$ is given by
   $$ W := \left(\frac{P-1}{P}\right)^P \left| \frac{\nabla G}{G}\right|^P \cdot \frac{|\gamma - 2 G|^{P-2}}{|\gamma - G|^P} \left( \gamma^2 + 2(P-2)G(\gamma - G) \right) $$
   with $\gamma := G(0) =  \int_0^\infty \frac{dt}{(\omega_n f(t))^{1/(P-1)}}$.
\end{remark}   
   
\begin{remark}
   In the special case of a Damek-Ricci space $X = X^{p,q}$, the asymptotics of $W$ for $r \to \infty$ can be evaluated more precisely by adapting the
   computations in \cite[p.153]{BDAGG17}: The density function of $X^{p,q}$ is given by
   $$ f(r) = 2^{p+q} (\sinh(r/2))^{p+q}(\cosh(r/2))^q, $$
   the mean curvature of horospheres is given by $h=\frac{p+2q}{2}$,
   and we obtain with $s = \sinh(t/2)$ and $\alpha = \frac{q}{2(P-1)}+\frac{1}{2}$,
   \begin{multline*}
       (2^{p+q} \omega_n)^{\frac{1}{P-1}} G(r) = \int_r^\infty (\sinh(t/2))^{-\frac{p+q}{(P-1)}} (\cosh(t/2))^{-\frac{q}{2(P-1)}} dt \\
       = 2 \int_{\sinh \frac{r}{2}}^\infty s^{-\frac{p+q}{P-1}} \left( 1+s^2 \right)^{-\alpha} ds = 2 \int_{\sinh\frac{r}{2}}^\infty s^{-\frac{p+q}{P-1} -2\alpha} \left( 1- \frac{\alpha}{s^2} + o\left( \frac{1}{s^3} \right)  \right) ds \\
       = \frac{P-1}{h} \left(\sinh\frac{r}{2}\right)^{-\frac{2h}{P-1}} - 2 \alpha \left( \frac{2h}{P-1} +2 \right)^{-1} \left(\sinh\frac{r}{2}\right)^{-\frac{2h}{P-1}-2} + o\left( \left(\sinh\frac{r}{2}\right)^{-\frac{2h}{P-1}-3}\right),
   \end{multline*}
   for $r \to \infty$.
   This implies that
   \begin{multline*}
       \left| \frac{G'(r)}{G(r)}  \right|^P \\ = \left( \coth \frac{r}{2}\right)^{- \frac{P}{P-1}\cdot q}\left|
       \frac{P-1}{h}  -2\alpha
       \left(\frac{2h}{P-1}+2 \right)^{-1}  \left( \sinh \frac{r}{2} \right)^{-2} + o\left( \left(\sinh\frac{r}{2}\right)^{-3}\right) \right|^{-P},
    \end{multline*}
    and therefore the following asymptotics of $W$ at infinity:
    \begin{multline*}
       W(r) \\ = \left( \frac{h}{P} \right)^P \left( \coth \frac{r}{2}\right)^{- \frac{P}{P-1}\cdot q}
       \left( 1 + 2\alpha P \frac{h}{2h+2P-2} \left(\sinh \frac{r}{2} \right)^{-2} + o\left( \sinh \frac{r}{2} \right)^{-3} \right).
   \end{multline*}
\end{remark}   
   

The corresponding $P$-Rellich inequality can be found in \cite[Theorem~7.3]{DP16}, and a quasi-linear version of the uncertainty principle follows by using the Hölder inequality properly (confer with Subsection~\ref{sec:Applications}).

\appendix

\section{Non-compact harmonic manifolds and Damek-Ricci spaces}\label{app:harmDR}

Let us start with the original definition of a harmonic manifold:
 A complete Riemannian manifold $(X,g)$ is harmonic if, for all $o \in X$, there exists a local non-constant radial harmonic function, that is, a function $g(x) = g(d(o,x))$, depending only on the distance to $o \in X$ and satisfying $\Delta g = 0$ on a small punctured neighbourhood of $o$. There are various other equivalent definitions of harmonicity which can be found, e.g., in  \cite[p. 224]{Willm93}, \cite[Prop. 6.21]{Besse78} or \cite[Th{\'e}or\`eme 4]{Rou03}. A very natural way to characterize harmonic manifolds (going back to Willmore) is to require that all harmonic functions satisfy the Mean Value Property. This characterization was given in the introduction.

It is known for any non-compact harmonic manifold that its volume density $f(r)$, $r=d(o,\cdot)$, is a strictly positive exponential polynomial, that is, a finite sum $\sum_{i=1}^k (p_i(r) \sin(\beta_i r) + q_i(r) \cos(\beta_i r))e^{\alpha_i r}$ (see \cite[Theorem 2]{Ni05}). Moreover, if a non-compact harmonic manifold of dimension $n$ has sub-exponential volume growth, then it must be isometric to the flat Euclidean space $\mathbb{R}^n$ (this follows from \cite{RS02}).

Let us now focus on an explicit family of non-compact harmonic manifolds, namely the Damek-Ricci spaces. They are particular Lie groups with left-invariant metrics. Their definition requires some preparation. 

A 2-step nilpotent Lie algebra $\mathfrak{n}$ with center $\mathfrak{z}$ and an inner product $\langle \cdot,\cdot \rangle: \mathfrak{n} \times \mathfrak{n} \to \mathbb{R}$ is called a \emph{Lie algebra of Heisenberg-type} if it has a orthogonal decomposition $\mathfrak{n} = \mathfrak{v} \oplus \mathfrak{z}$ and if the linear maps $J_Z: \mathfrak{v} \to \mathfrak{v}$ for each $Z \in \mathfrak{z}$, defined by
$$ \langle J_Z V,V' \rangle = \langle Z,[V,V'] \rangle $$
satisfy 
\begin{equation} \label{eq:JZ}
J_Z^2(V) = - \Vert Z \Vert^2 \, V. 
\end{equation}
A Heisenberg-type Lie algebra $\mathfrak{n}$ comes with the two parameters $p,q \ge 1$ which are the dimensions $p = {\rm{dim}}(\mathfrak{v})$ and $q = {\rm{dim}}(\mathfrak{z})$. The above conditions restrict the possible pairs $(p,q)$, for example, it follows from \eqref{eq:JZ} that $p$ needs to be an even number. 
Let $\mathfrak{a}$ be a $1$-dimensional Lie algebra generated by $H \in \mathfrak{a}$, that is, $\mathfrak{a} = \mathbb{R} H$. We extend a Lie algebra $\mathfrak{n}$ of Heisenberg-type
to a solvable Lie algebra 
\begin{equation} \label{eq:defs} 
\mathfrak{s} = \mathfrak{n} \oplus \mathfrak{a}, 
\end{equation}
of dimension $n = p+q+1$ by setting
$$ [H,V] = \frac{1}{2}V \quad \text{and} \quad [H,Z] = Z
$$
for all $V \in \mathfrak{v}$ and $Z \in \mathfrak {z}$. We also extend the inner product on $\mathfrak{n}$ to an
inner product $\langle \cdot,\cdot \rangle: \mathfrak{s} \times \mathfrak{s} \to \mathbb{R}$, by setting $\Vert H \Vert = 1$ and requiring that the decomposition \eqref{eq:defs} is orthogonal. We can now give the definition of a Damek-Ricci space.

\begin{definition}
  A \emph{Damek-Ricci} space $X^{p,q}$ is the unique connected and simply connected Lie group $S = NA$ associated to the solvable extension $\mathfrak{s}$ of a Heisenberg-type Lie algebra $\mathfrak{n}$ with parameters $p,q$ by $\mathfrak{a} = \mathbb{R} H$, as described above, and equipped with a left invariant metric such that the Riemannian metric of $T_eS$ agrees with the inner product on $\mathfrak{s}$ under the canonical identification $T_eS \cong \mathfrak{s}$, where $e \in S$ denotes the neutral element of $S$. The space $X^{p,q}$ has therefore the dimension $n = p+q+1$. 
\end{definition}





Finally, let us briefly explain how the real hyperbolic space ${\mathbb{H}}^n$ can be described as a solvabe Lie group with left-invariant metric, even though this space is associated to a Lie algebra $\mathfrak{s} = \mathfrak{n} \oplus \mathfrak{a}$ where $\mathfrak{n}$ is not $2$-step nilpotent but abelian. For that reason, real hyperbolic spaces are not contained in the class of Damek-Ricci spaces.

\begin{example}
  The upper half space model of the real hyperbolic space $\mathbb{H}^n$ is given by $\mathbb{H}^n = \{ (\mathbf{x},y): \mathbf{x} \in \mathbb{R}^{n-1}, y > 0 \}$
  with Riemannian metric
  $$ g_{\mathbb{H}^n}(v,w) = \frac{\langle v,w \rangle_0}{y^2} \quad \text{for $v,w \in T_{({\mathbf{x}},y)} \cong \mathbb{R}^n$,} $$
  where $\langle \cdot,\cdot \rangle_0$ denotes the Euclidean inner product of $\mathbb{R}^n$. For this space we have
  $\lambda_0(\mathbb{H}^n) = \frac{(n-1)^2}{4}$.
  
  For the description of $\mathbb{H}^n$ as a solvable Lie group $S = NA$, we start with the Lie algebras
  $$ \mathfrak{n} = \left\{ n({\mathbf{x}}): {\mathbf{x}} \in {\mathbb{R}}^{n-1} \right\} \quad \text{and} \quad \mathfrak{a} = \left\{ {\rm{diag}}(0,0,\dots,0,t,-t): 
  t \in {\mathbb{R}} \right\}, $$ 
  where
  $$ n({\mathbf{x}}) = \begin{pmatrix} 0_{n-1} & 0 & - {\mathbf{x}}^\top \\ {\mathbf{x}} & 0 & 0 \\ 0 & 0 & 0 \end{pmatrix}. $$
  Let $H = {\rm{diag}}(0,0,\dots,0,1,-1)$.
  It is easy to see that $[H,X] = HX-XH = X$ and $[X,Y] = XY-YX = 0$ for all $X,Y \in \mathfrak{n}$. An inner product on the solvable Lie algebra $\mathfrak{s} = \mathfrak{n} \oplus \mathfrak{a}$ with Lie bracket $[Z_1,Z_2] = Z_1 Z_2 - Z_2 Z_1$ is given by the requirements that $\mathfrak{n}$ and $\mathfrak{a}$ are perpendicular, 
  $\Vert H \Vert =1$, and
  $$ \langle n({\mathbf{x}}),n({\mathbf{y}}) \rangle = \langle {\mathbf{x}},{\mathbf{y}} \rangle_0, $$
  where $\langle \cdot,\cdot \rangle_0$ denotes the Euclidean inner product on $\mathbb{R}^{n-1}$. The map 
  \begin{eqnarray*} \mathfrak{s} = \mathfrak{n} \oplus \mathfrak{a} &\to& S = NA, \\ (n({\mathbf{x}}),tH) &\mapsto& \exp(n({\mathbf{x}})) \exp(tH) = \begin{pmatrix} {\rm{Id}}_{n-1} & 0 & -e^{-t} {\mathbf{x}}^\top \\ {\mathbf{x}} & e^t & - \frac{1}{2} e^{-t} \Vert {\mathbf{x}} \Vert_0^2 \\ 0 & 0 & e^{-t} \end{pmatrix}
  \end{eqnarray*}
  is a bijection. If we identify $S=NA$ with
  $\mathbb{R}^{n-1} \times \mathbb{R}^+$ via
  $$ \exp(n({\mathbf{x}})) \exp(tH) \mapsto ({\mathbf{x}},e^t), $$
  the Lie group multiplication in $S$ takes the form
  $$ ({\mathbf{y}},e^s) \cdot ({\mathbf{x}},e^t) = ({\mathbf{y}}+e^s{\mathbf{x}},e^{s+t}). $$
  In other words, $S = NA$ can be identified with a semidirect product $\mathbb{R}^{n-1} \rtimes \mathbb{R}^+$, where the group operation in $\mathbb{R}^{n-1}$ is addition
  and the group operation in $\mathbb{R}^+$ is multiplication. 
  
  We transfer the inner product of $\mathfrak{s}$ to an inner product on $T_eS$ via the identification $Z \mapsto \frac{d}{dt}\vert_{t=0} \exp(t Z)$ and extend it left-invariantly to a Riemannian metric $g_S$ on $S$. Then the map
  $$ \exp(n({\mathbf{x}})) \exp(sH) \mapsto ({\mathbf{x}},e^s) \in \mathbb{H}^n $$
  is an isometry between $(S,g_S)$ and $({\mathbb{H}}^n,g_{{\mathbb{H}}^n})$, which implies that $(S,g_S)$ is a model of the $n$-dimensional real hyperbolic space.
\end{example}

\noindent\textbf{Acknowledgements:} F.~F. wants to thank the University of Durham for their kind hospitality during a stay where the main idea of the paper was worked out. Moreover, F.~F. wants to thank Noema Nicolussi for pointing out paper \cite{BGGP20}, and also the Heinrich-Böll-Stiftung for the support.

\end{document}